\RequirePackage{fix-cm}
\documentclass[smallextended]{svjour3}       
\smartqed  
\usepackage{graphicx}%
\usepackage{multirow}%
\usepackage{amsmath,amssymb,amsfonts}%
\usepackage{mathrsfs}%
\usepackage[title]{appendix}%
\usepackage{xcolor}%
\usepackage{textcomp}%
\usepackage{manyfoot}%
\usepackage{booktabs}%
\usepackage{algorithm}%
\usepackage{algorithmicx}%
\usepackage{algpseudocode}%
\usepackage{listings}%
\usepackage{enumerate}
\usepackage{xurl}
\DeclareMathOperator*{\argmin}{argmin}

\DeclareMathOperator*{\epi}{epi}

\def\approxleq{ \kern3pt \mbox{\raisebox{.6ex}{$<$}} \kern-8pt
	\mbox{\raisebox{-.6ex}{$\sim$}} \kern5pt}

\def\bR{\mathbb{R}} \def\bS{\mathbb{S}}

\def\cA{{\cal A}}

\begin{document}

\title{A Level Set Method with Secant Iterations for the Least-Squares Constrained Nuclear Norm Minimization \thanks{Chiyu Ma is supported in part by the Research Center for Intelligent Operations Research at the Department of Applied Mathematics of the Hong Kong Polytechnic University and the Hong Kong Research Grants Council under the Senior Research Fellow Scheme (No. SRFS2223-5S02). Defeng Sun is supported in part by the Hong Kong Research Grants Council under GRF Projects No. 15309625 and No. 15305324.}}

\titlerunning{A Level Set Method for Least-Squares Constrained Nuclear Norm Minimization}

\author{Chiyu Ma \and Jiaming Ma \and Defeng Sun}

\authorrunning{Chiyu Ma, Jiaming Ma, and Defeng Sun}

\institute{
Chiyu Ma \at
Department of Applied Mathematics, The Hong Kong Polytechnic University, Hung Hom, Hong Kong \\
\email{chiyu.ma@connect.polyu.hk}
\and
Jiaming Ma \at
Department of Applied Mathematics, The Hong Kong Polytechnic University, Hung Hom, Hong Kong \\
\email{jiaming.ma@connect.polyu.hk}
\and
Defeng Sun, Corresponding Author
\at
Department of Applied Mathematics, The Hong Kong Polytechnic University, Hung Hom, Hong Kong \\
\email{defeng.sun@polyu.edu.hk}
}

\date{Received: date / Accepted: date}

\maketitle

\begin{abstract}
We present an efficient algorithm for least-squares constrained nuclear norm minimization, a computationally challenging problem with broad applications. Our approach combines a level set method with secant iterations and a proximal generation method. As a key theoretical contribution, we establish the nonsingularity of the Clarke generalized Jacobian for a general class of projection norm functions over closed convex sets. This property and the (strong) semismoothness of our value function yield fast local convergence of the secant method. For the resulting nuclear norm regularized subproblems, we develop a proximal generation method that exploits low-rank structures without compromising convergence. Extensive numerical experiments demonstrate the superior performance of our approach compared to state-of-the-art methods.

\end{abstract}

\section{Introduction}\label{sec1}

In this paper, we consider the following least-squares constrained nuclear norm minimization:
\begin{equation}\label{eq:lscnnm}\tag{CP($\varrho$)}
	\min _{X \in \mathbb{R}^{m\times n}}\|X\|_* \quad\text{ s.t. }\|\cA X-b\| \leqslant \varrho,
\end{equation} 
where $\|\cdot\|_*$ denotes the nuclear norm, $\|\cdot\|$ represents the Euclidean norm for vectors (or the Frobenius norm for matrices), $b\in\bR^p$, $\cA$ is a linear mapping from $\bR^{m\times n}$ to $\bR^p$, and $0<\varrho<\|b\|$. We assume that the problem \eqref{eq:lscnnm} is feasible and $\|\cA^*b\|_2>0$. It is observed that the least-squares constraint is active at any optimal solution. The least-squares constraint quantifies the fitting accuracy of the target model, while the nuclear norm in the objective function aims to enhance the low-rank nature of the solution.

A more prevalent model is the nuclear norm regularized least-squares problem:
\begin{equation}\label{eq:nnrls}\tag{LS($\lambda$)-P}
	\min _{X \in \mathbb{R}^{m\times n}}f_\lambda(X):=\lambda\|X\|_* +\frac{1}{2}\|\cA X-b\|^2 .
\end{equation}
This formulation is computationally easier, and methods such as Accelerated Proximal Gradient (APG) \cite{Toh2010}, Alternating Direction Method of Multipliers (ADMM) \cite{Fazel2013}, and Proximal Point Algorithm (PPA) \cite{Jiang2013,Jiang2014} have been developed for it. When dealing with large-scale matrices for which SVD is intractable, low-rank decomposition methods \cite{Burer2003,Burer2005,Lee2024} serve as the de facto solution.

Compared to \eqref{eq:nnrls}, the constrained formulation \eqref{eq:lscnnm} provides enhanced noise control through the tuning parameter $\varrho$, which makes it more practical. However, the least-squares constraint and the nonpolyhedral nuclear norm introduce additional computational challenges. Efficient solvers for this problem would bridge the gap between theoretical guarantees and practical usability.

The level set method \cite{VanDenBerg2009,VandenBerg2011,Aravkin2019} is a notable approach to address the least-squares constrained problem \eqref{eq:lscnnm}. This kind of level set method exchanges the objective and constraints. Specifically, it transforms the original least-squares constrained problems \eqref{eq:lscnnm} into a univariate nonlinear equation:
	$$\phi(\tau):=\|\mathcal{A}X(\tau)-b\|=\varrho, \text{ where } X(\tau)\in\arg\min_X \left\lbrace\frac{1}{2}\|\mathcal{A}X-b\|^2 \mid \|X\|_*\leqslant\tau \right\rbrace.$$
Although this nuclear norm constrained least-squares problem can be tackled using the projected gradient method, a conventional choice in prior studies, this approach may not be computationally efficient. 

A different approach involves a variation of the level set method that solves a series of regularized subproblems \cite{Li2018}. It reformulates problem \eqref{eq:lscnnm} into the following univariate nonlinear equation:
\begin{equation}\label{eq:une}
	\varphi(\lambda):=\|\cA X^*(\lambda)-b\|=\varrho,
\end{equation}
where $X^*(\lambda)$ is an optimal solution of \eqref{eq:nnrls}. This univariate function $\varphi$ is referred to as the value function. The nuclear norm regularized least-squares problem \eqref{eq:une} is typically more computationally tractable than the constrained one \eqref{eq:lscnnm}.

For the sparse optimization variant of \eqref{eq:lscnnm} with a polyhedral objective, the sieving-based secant method \cite{Li2024}, which combines the level set method \cite{Li2018}, adaptive sieving \cite{Yuan2025}, and the secant method \cite{Potra1998}, is computationally efficient. However, our problem differs fundamentally due to its nonpolyhedral objective and low-rank pursuit nature. Although existing matrix screening rules \cite{Hsieh2014,Zhou2015,Zhong2023} attempt to leverage the low-rank structure of \eqref{eq:nnrls} solutions, their inherent conservatism prevents direct application to our formulation. These limitations highlight the critical need for new computational frameworks that can simultaneously address nonpolyhedral objectives while effectively exploiting low-rank structures.

In this paper, we propose a proximal generation based level set method with secant iterations, specifically designed for the least-squares constrained nuclear norm minimization problem \eqref{eq:lscnnm}. Building upon the level set framework \cite{Li2018} and the secant method \cite{Potra1998}, our approach achieves fast convergence by ensuring superlinear convergence rates for the value function in the outer level set iterations. For the sequence of regularized subproblems, the proposed proximal generation method fully exploits the low-rank structure of solutions to reduce computational costs without compromising global convergence. Together, these features enable our method to achieve superior performance compared to existing approaches.

We summarize the main contributions of this paper as follows:
\begin{enumerate}
	\item We establish the nonsingularity of the Clarke generalized Jacobian $\partial \psi(\cdot)$ for the projection norm function $\psi(\lambda):=\|\Pi_{\lambda C}(b)\|$ for $0<\lambda<\Upsilon(b\mid C)$, where $C$ is a closed convex set and $0<\Upsilon(b\mid C)<+\infty$. The value function $\varphi(\lambda)$ in \eqref{eq:une} is a specific instance of $\|\Pi_{\lambda C}(b)\|$. We also provide a sufficient condition for the strong semismoothness of $\varphi(\lambda)$. Consequently, the secant method for \eqref{eq:une} can achieve a fast convergence rate.
	\item To efficiently solve a sequence of regularized problems \eqref{eq:nnrls}, we design a proximal generation method that leverages the low-rank decomposition of the nuclear norm to reduce dimensionality and employs proximal gradient techniques to ensure global convergence. This is the first attempt to extend the adaptive sieving strategy \cite{Yuan2025} to nonpolyhedral and even nonconvex cases. Combined with the secant method, our proximal generation based level set approach boosts efficiency in the least-squares constrained nuclear norm minimization problems.
	\item We develop several practical implementation techniques and conduct extensive numerical experiments on matrix regression and matrix completion. Compared with state-of-the-art algorithms, our proximal generation based level set method with secant iterations exhibits superior numerical performance. 
\end{enumerate}

The organization of the paper is as follows. Section 2 provides some necessary preliminaries. In Section 3, we propose our level set method with secant iterations. We analyze the nonsingularity of $\partial \varphi(\lambda)$ and the (strong) semismoothness of $\varphi(\lambda)$. Section 4 delves into the introduction of a proximal generation method tailored for subproblems within the level set paradigm. We also discuss the inexact version of the secant method and the proximal generation method in these sections. In Section 5, we conduct numerical experiments focused on low-rank matrix regressions and completions. Section 6 concludes the paper.
\paragraph{Notations.} Our notation adheres to standard conventions, with specific symbols clarified upon initial use where necessary. In the following, we list some commonly used ones. The set of positive real numbers is denoted by $\bR_{++}$. The epigraph of the spectral norm (a matrix cone) is denoted by $\epi \|\cdot\|_2:=\left\lbrace(t,X)\mid t\geqslant \|X\|_2 \right\rbrace $. The dual cone of the epigraph of the spectral norm is the epigraph of the nuclear norm $\epi \|\cdot\|_*:=\left\lbrace(t,X)\mid t\geqslant \|X\|_* \right\rbrace $. For a proper closed convex function $f:\bR^n\to \left( -\infty,\infty\right] $, its proximal mapping is defined as $\operatorname{prox}_{f(\cdot)}(z)=\argmin_x f(x)+\frac{1}{2}\|x-z\|^2$ where the subscript $(\cdot)$ could be omitted when there is no ambiguity. The indicator function of a closed nonempty convex set $C$ is denoted as \[\delta_C(x):=\begin{cases}
	0& \text{ if $x\in C$;}\\
	+\infty & \text{ else. }
\end{cases}\]
The proximal mapping of an indicator function $\delta_C(\cdot)$ is the projection onto $C$, denoted as $\Pi_C(z)$. The gauge function of a convex set $C$ is defined as $\Upsilon(x\mid C):=\inf\left\lbrace \lambda\geqslant0\mid x\in\lambda C\right\rbrace $. 

\section{Preliminaries}

Let \(\mathcal{X}\) and \(\mathcal{Y}\) be two real Euclidean spaces, and let
\(\mathcal{L}(\mathcal{X},\mathcal{Y})\) denote the space of linear operators
from \(\mathcal{X}\) to \(\mathcal{Y}\).
Let \(F:\mathcal{O}\to\mathcal{Y}\) be a locally Lipschitz continuous function,
where \(\mathcal{O}\subset \mathcal{X}\) is open.
By Rademacher's theorem, \(F\) is Fr\'echet differentiable almost everywhere on
\(\mathcal{O}\).
Denote by \(D_F\) the set of differentiable points of \(F\) on \(\mathcal{O}\),
and by \(F'(x)\) the Jacobian of \(F\) at \(x\in D_F\).
The B-subdifferential of \(F\) at \(x\in\mathcal{O}\) is defined by
\begin{equation*}
	\partial_B F(x)
	:= \left\{\, V\in \mathcal{L}(\mathcal{X},\mathcal{Y}) \mid \exists \{x^k\}\subset D_F,\ x^k \to x,\ F'(x^k)\to V \,\right\}.
\end{equation*}
The Clarke generalized Jacobian of \(F\) at \(x\in\mathcal{O}\) is defined as
\(
\partial F(x):=\operatorname{conv}\bigl(\partial_B F(x)\bigr).
\)
Both \(\partial_B F(\cdot)\) and \(\partial F(\cdot)\) are compact-valued and
upper semicontinuous. For any convex function $f:\mathbb{R}^n \to \mathbb{R}$, the Clarke generalized Jacobian is equivalent to subdifferential in the context of convex analysis \cite[Proposition 2.2.7]{Clarke1990}.

Here we review the concept of (strong) G-semismoothness.
\begin{definition}\label{def:gss}
	Let $F:\mathcal{O}\subset\mathcal{X}\to\mathcal{Y}$ be a locally Lipschitz continuous function on the open set $\mathcal{O}$. Function $F$ is said to be G-semismooth at $x\in\mathcal{O}$ if for any $h\to0$ and $V\in\partial F(x+h)$, 
	\begin{equation}\label{eq:def-gssm}
		F(x+h)-F(x)-Vh=o(\|h\|).
	\end{equation}
	Let $\gamma$ be a positive constant. Function $F$ is said to be $\gamma-$order G-semismooth at $x\in\mathcal{O}$ if for any $h\to0$ and $V\in\partial F(x+h)$, 
	\begin{equation}\label{eq:def-gammassm}
		F(x+h)-F(x)-Vh=O(\|h\|^{1+\gamma}).
	\end{equation}
	Function $F$ is said to be G-semismooth ($\gamma-$order G-semismooth) on $\mathcal{O}$ if it is G-semismooth ($\gamma-$order G-semismooth) at any point $x\in\mathcal{O}$.
	Function $F$ is called strongly G-semismooth at $x$ if it is $1-$order G-semismooth at $x$.
\end{definition}
Another related concept is the directional differentiability. We say that $F$ is directionally differentiable at $x\in\mathcal{O}$ if the limit \[F'(x;h):=\lim_{t\downarrow0}\frac{F(x+th)-F(x)}{t}\] exists for every direction $h\in\mathcal{X}$. For a univariate function $\varphi(x)$, its two lateral derivatives at $x$ are defined as $\bar{d}^{-}:=-\varphi^{\prime}(\bar{x} ;-1)$ and $\bar{d}^{+}:=\varphi^{\prime}(\bar{x} ; 1)$. According to \cite{Shapiro1990}, for a locally Lipschitz continuous function $F:\mathcal{X}\to\mathcal{Y}$ that is directionally differentiable at $x$, we have for any $h\to0$
\begin{equation}\label{eq:dd-bd}
	F(x+h)-F(x)-F'(x;h)=o(\|h\|).
\end{equation}  
If function $F:\mathcal{O}\subset\mathcal{X}\to\mathcal{Y}$ is both G-semismooth and directionally differentiable at $x\in\mathcal{O}$, it is said to be semismooth at $x\in\mathcal{O}$. The concept of semismoothness on a set is as defined in Definition \ref{def:gss}. Similarly to the $\gamma-$order G-semismooth \eqref{eq:def-gammassm}, 
we can strengthen the directional differentiability as follows.
\begin{definition}
	Let $\gamma$ be a positive constant. Function $F:\mathcal{O}\subset\mathcal{X}\to\mathcal{Y}$ is said to be $\gamma-$order directionally differentiable at $x\in\mathcal{O}$ if for any $h\to0$ 
	\begin{equation}\label{eq:cbd}
		F(x+h)-F(x)-F'(x;h)=O(\|h\|^{1+\gamma}).
	\end{equation}  
	Function $F$ is said to be ($\gamma-$order) directionally differentiable on $\mathcal{O}$ if it is ($\gamma-$order) directionally differentiable at any point $x\in\mathcal{O}$.
	Function $F$ is called calmly B-differentiable at $x$ if \eqref{eq:cbd} holds with $\gamma=1$.
\end{definition}
If function $F:\mathcal{O}\subset\mathcal{X}\to\mathcal{Y}$ is both $\gamma-$order G-semismooth and $\gamma-$order directionally differentiable at $x\in\mathcal{O}$, it is said to be $\gamma-$order semismooth at $x\in\mathcal{O}$.
If function $F$ is both strongly G-semismooth and calmly B-differentiable at $x\in\mathcal{O}$, it is said to be strongly semismooth at $x\in\mathcal{O}$.
\begin{proposition}\label{prop:1ss}
	Let $F:\mathcal{O}\subset\mathcal{X}\to\mathcal{Y}$ be a locally Lipschitz continuous function on the open set $\mathcal{O}$ and $x\in\mathcal{O}$. The following statements are equivalent for any $\gamma>0$:
	
	(i) $F$ is $\gamma-$order semismooth at $x$, i.e., \eqref{eq:def-gammassm} and \eqref{eq:cbd} hold;
	
	(ii) for any $h\to 0$ and $V\in\partial F(x+h)$, it holds
	\begin{equation}\label{eq:def-gammass1}
		Vh-F'(x;h)=O(\|h\|^{1+\gamma});
	\end{equation}
	
	(iii) for any $h\to 0$ and $x+h\in D_F$, it holds
	\begin{equation}\label{eq:def-gammass2}
		F'(x+h;h)-F'(x;h)=O(\|h\|^{1+\gamma}).
	\end{equation} 
\end{proposition}
\begin{proof}
	The implications (i)$\Rightarrow$(ii) and (ii)$\Rightarrow$(iii) are obvious. By the derivation of \cite[Theorem 7.4.3]{Facchinei2003} with the Lebesgue integral, we have (iii)$\Rightarrow$\eqref{eq:cbd}. With \cite[Theorem 3.7]{Sun2002}, we can get \eqref{eq:def-gammassm} from \eqref{eq:def-gammass2} and \eqref{eq:cbd}. Therefore, we have (iii)$\Rightarrow$(i). This completes the proof.
\end{proof}

Let $G:\mathcal{X}\to\mathcal{Y}$ be a twice continuously differentiable mapping and let $K\subset \mathcal{Y}$ be a closed convex set. The tangent cone of $K$ at $y$ is denoted as $T_K(y)$ and defined as $T_K(y):=\{d \in \mathcal{Y}\mid \operatorname{dist}(y+t d, K)=o(t), t \geq 0\}$. The linearity space of a closed convex cone $C$ is denoted as $\operatorname{lin}(C)$, with $\operatorname{lin}(C)=C\cap(-C)$. The nondegeneracy of the constraint system $G(x)\in K$ is defined as follows.
\begin{definition}
	A feasible point $\bar{x}$ of constraint system $G(x)\in K$ is said to be constraint nondegenerate if
	\begin{equation}
		G'(\bar{x})\mathcal{X}+\operatorname{lin}\left( T_K(G(\bar{x})) \right)=\mathcal{Y}. 
	\end{equation}
\end{definition}
\begin{definition}
	A function $\Phi: \mathcal{O} \subseteq \mathcal{X} \rightarrow \mathcal{X}$ is said to be a locally Lipschitz homeomorphism near $x \in \mathcal{O}$ if there exists an open neighborhood $\mathcal{N} \subseteq \mathcal{O}$ of $x$ such that the restricted map $\Phi\mid_{\mathcal{N}}: \mathcal{N} \rightarrow \Phi(\mathcal{N})$ is Lipschitz continuous and bijective, and its inverse is also Lipschitz continuous.
\end{definition} 
\begin{proposition}\label{prop:in-fun-th}
	Let function $\Phi: \mathcal{O} \subseteq \mathcal{X} \rightarrow \mathcal{X}$ be a locally Lipschitz homeomorphism near $x_0 \in \mathcal{O}$. Let $\gamma$ be a positive constant. Then there exists an open neighborhood $\mathcal{N} \subseteq \mathcal{O}$ of $x_0$ such that:
	
	(i) $\Phi$ is G-semismooth ($\gamma-$order G-semismooth) at some point $\bar{x} \in \mathcal{N}$ if and only if $\Phi^{-1}$, the local inverse mapping of $\Phi$ near $x_0$, is G-semismooth ($\gamma-$order G-semismooth) at $\bar{y}:=\Phi(\bar{x})$;
	
	(ii) $\Phi$ is directionally differentiable ($\gamma-$order directionally differentiable) at some point $\bar{x} \in \mathcal{N}$ if and only if $\Phi^{-1}$, the local inverse mapping of $\Phi$ near $x_0$, is directionally differentiable ($\gamma-$order directionally differentiable) at $\bar{y}:=\Phi(\bar{x})$.
\end{proposition}
\begin{proof}
	The inverse result for G-semismoothness follows from
	\cite[Theorem~2]{Meng2005}. The proof for the $\gamma$-order
	G-semismoothness is the same as that for the strong
	G-semismoothness in \cite[Theorem~2]{Meng2005}. It remains to
	establish the result for $\gamma$-order directional
	differentiability.
	
	Fix $\bar{x}\in\mathcal N$ and set
	\[
	\bar{y}:=\Phi(\bar{x}),
	\qquad
	\Psi(\cdot):=\Phi'(\bar{x};\cdot).
	\]
	Suppose first that $\Phi$ is $\gamma$-order directionally
	differentiable at $\bar{x}$. Since this implies ordinary directional
	differentiability, Kummer's inverse result
	\cite[Lemma~2]{Kummer1992} shows that $\Psi$ is a Lipschitz
	homeomorphism and
	$(\Phi^{-1})'(\bar{y};\cdot)=\Psi^{-1}(\cdot)$.
	For any $y\in\Phi(\mathcal N)$ with $y\to\bar y$, define
	\[
	x:=\Phi^{-1}(y)
	\quad \text{and}\quad 
	r
	:=
	\Phi(x)-\Phi(\bar{x})-\Psi(x-\bar{x}).
	\]
	Then $x\to\bar{x}$ and, by the $\gamma$-order directional
	differentiability of $\Phi$ at $\bar{x}$,
	\[
	\|r\|
	=
	O\bigl(\|x-\bar{x}\|^{1+\gamma}\bigr).
	\]
	Moreover,
	\[
	y-\bar{y}
	=
	\Psi(x-\bar{x})+r.
	\]
	Consequently,
	\begin{equation}\label{eq:imfunthm_cbd}
		\begin{aligned}
			&
			\left\|
			\Phi^{-1}(y)-\Phi^{-1}(\bar{y})
			-
			(\Phi^{-1})'(\bar{y};y-\bar{y})
			\right\|
			\\
			={}&
			\left\|
			x-\bar{x}-\Psi^{-1}(y-\bar{y})
			\right\|
			\\
			={}&
			\left\|
			\Psi^{-1}\bigl(\Psi(x-\bar{x})\bigr)
			-
			\Psi^{-1}\bigl(\Psi(x-\bar{x})+r\bigr)
			\right\|
			\\
			\leqslant{}&
			L_{\Psi^{-1}}\|r\|
			=
			O\bigl(\|x-\bar{x}\|^{1+\gamma}\bigr)
			=
			O\bigl(\|y-\bar{y}\|^{1+\gamma}\bigr).
		\end{aligned}
	\end{equation}
	Here the last equality follows from the local Lipschitz continuity
	of $\Phi^{-1}$, which gives
	\[
	\|x-\bar{x}\|
	=
	O(\|y-\bar{y}\|).
	\]
	Hence, $\Phi^{-1}$ is $\gamma$-order directionally differentiable
	at $\bar{y}$. The converse implication follows by applying the same
	argument to $\Phi^{-1}$. This completes the proof.
	\end{proof}

\section{A level set method with secant iterations for \eqref{eq:lscnnm}}
In this section, we first demonstrate several properties of the value function $\varphi(\lambda)$ in the univariate equation \eqref{eq:une} and the level set method from \cite{Li2018}. Then we prove the nonsingularity of the Clarke generalized Jacobian of a class of univariate functions including $\varphi(\lambda)$. Finally, we analyze the semismoothness of $\varphi(\lambda)$ and develop a secant method for \eqref{eq:une}.

\subsection{A level set method with the value function $\varphi(\lambda)$}
The dual problem of the regularized problem \eqref{eq:nnrls} takes the following form:
\begin{equation}\label{eq:nnrls-d}\tag{LS($\lambda$)-D}
	\min _{y \in \mathbb{R}^{p},Z\in\bR^{m\times n}}\frac{1}{2}\|y\|^2-b^Ty\quad\text{s.t. } \cA^*y= \lambda Z, \|Z\|_2\leqslant 1 .
\end{equation}
The KKT conditions of \eqref{eq:nnrls} and \eqref{eq:nnrls-d} are:
\begin{equation}\label{eq:nnrls-kkt}
	Z\in \partial \|X\|_*,\quad y-b+\cA X=0,\quad \text{and }\cA^*y=\lambda Z. 
\end{equation}
The dual problem can be equivalently formulated as 
\begin{equation}\label{eq:nnrls-d2}
	\min _{y \in \mathbb{R}^{p}}\frac{1}{2}\|y\|^2-b^Ty\quad \text{ s.t. } \lambda^{-1}y\in Q ,
\end{equation}
where $Q:=\{z\in\bR^p\mid \cA^*z\in\partial\|0\|_*\}=\{z\in\bR^p\mid \|\cA^*z\|_2\leqslant1\}$. 
We denote the primal-dual solutions of \eqref{eq:nnrls} and \eqref{eq:nnrls-d} as $(X(\lambda),y(\lambda),Z(\lambda))$. With \eqref{eq:nnrls-d2}, the dual solution $y(\lambda)$ can be explicitly expressed as $\lambda\Pi_Q(\lambda^{-1}b)=\Pi_{\lambda Q}(b)$. Hence, the value function $\varphi(\lambda)$ satisfies 
\begin{equation}\label{eq:valuef}
	\varphi(\lambda)=\|\cA X(\lambda)-b\|=\|y(\lambda)\|=\|\Pi_{\lambda Q}(b)\|.	
\end{equation}

To justify the level set method for \eqref{eq:lscnnm}, we demonstrate the following properties of this value function $\varphi(\lambda)$. 
\begin{proposition} \label{prop:valuefunction}
	Suppose that $0<\|\cA^*b\|_2$. The following statements hold:
	\begin{enumerate}[(i)]
		\item For any $\lambda\geqslant0$, the quantity $\varphi(\lambda)$ is independent of the particular choice of $X(\lambda)$;
		\item For any $\lambda>\|\cA^*b\|_2$, $y(\lambda)=b$ and $X(\lambda)=0$;
		\item For any $0< \lambda_1<\lambda_2\leqslant \|\cA^*b\|_2$, $\varphi(\lambda_1)<\varphi(\lambda_2)$.
	\end{enumerate}
\end{proposition}
The proof of Proposition \ref{prop:valuefunction} can be found in \cite[Proposition 1]{Li2018} and \cite[Proposition 3.1, Proposition 3.3]{Li2024}. Note that $Q$ is not necessarily a polyhedral set. The bound $\|\cA^*b\|_2$ is actually the value of the gauge function of $Q$ at $b$, namely $\Upsilon(b\mid Q)$.

With the monotonicity of the value function $\varphi(\lambda)$, the equivalence between the least-squares constrained problem \eqref{eq:lscnnm} and the univariate nonlinear equation \eqref{eq:une} is established, which justifies the following level set method. 
\begin{algorithm}[!h]
	\caption{A level set method for \eqref{eq:lscnnm}}
	\label{alg:levelset}
	\begin{algorithmic}[1]
		\Require Initial points $\lambda^{-1}$, $\lambda^0$, $\lambda^1$ and bounds $\lambda_{\text{lo}}$, $\lambda_{\text{up}}$ such that $0 \leqslant \lambda_{\text{lo}} \leqslant \lambda^{-1} < \lambda^0 \leqslant \lambda_{\text{up}} \leqslant \|\cA^*b\|_2$ and $\lambda_{\text{lo}}\leqslant\lambda^1\leqslant\lambda_{\text{up}}$.
		\For{$k=1,2\ldots$}
		\State Compute
		\begin{equation}\label{eq:subprob_of_level}
			\tag{LS($\lambda^k$)-P}
			X^{k}=\argmin_{X\in\bR^{m\times n}}\lambda^k\|X\|_* +\frac{1}{2}\|\cA X-b\|^2.
		\end{equation}
		\State Compute $\varphi(\lambda^k)=\|\cA X^k-b\|$.
		\State Update \( \lambda^{k+1} \) using either a bisection or a secant step.
		\EndFor
		
	\end{algorithmic}
\end{algorithm}

Algorithm \ref{alg:levelset} was originally proposed in \cite{Li2018} and introduces a novel approach compared to traditional level set methods \cite{Aravkin2019,VanDenBerg2009,VandenBerg2011}. The key advantage lies in the subproblem \eqref{eq:subprob_of_level}, which is easier to solve. Moreover, its regularized least-squares structure enables the use of a proximal generation method, as discussed in the next section.
\subsection{Nonsingularity of the Clarke generalized Jacobian $\partial \varphi(\lambda)$}
In this subsection, we investigate the nonsingularity of the Clarke generalized Jacobian $\partial \varphi(\lambda)$, which plays an important role in the superlinear convergence rate of the secant method. 

We consider a broader univariate function 
\begin{equation}\label{eq:pjsq}
	\psi(\lambda):=\|\lambda\Pi_{C}(\lambda^{-1}b)\|=\|\Pi_{\lambda C}(b)\|,
\end{equation}
where $b\in\bR^p$ is a constant vector and $C\subset\bR^p$ is a closed convex set such that $0\in C$ and $0<\Upsilon(b\mid C)<+\infty$. The value function $\varphi(\lambda)$ is an example of \eqref{eq:pjsq}. In addition, according to the Lipschitz continuity of the projection onto the closed convex set $C$ and the Rademacher theorem, $\psi(\lambda)$ is differentiable almost everywhere on $(0,\Upsilon(b\mid C))$.

We first state some technical results for $\psi(\lambda)$.
\begin{lemma}\label{lm:neq0}
	For any $\lambda\in(0,\Upsilon(b\mid C))$, we have $$\Pi_{\lambda C}(b)\neq0,\quad b-\Pi_{\lambda C}(b)\neq0\quad\text{and}\quad\left\langle\Pi_{\lambda C}(b),  b-\Pi_{\lambda C}(b)\right\rangle>0.$$
\end{lemma}
\begin{proof}
	Since $\Upsilon(b\mid C)>0$ and $0\in C$, it follows that $b\neq0$. We will prove the above three inequalities by contradiction.
	
	If $\Pi_{\lambda C}(b)=0$, then \[\left\langle \Pi_{\lambda C}(b)-\lambda {\Upsilon(b\mid C)}^{-1}b,b-\Pi_{\lambda C}(b) \right\rangle=-\lambda {\Upsilon(b\mid C)}^{-1}\|b\|^2<0. \]
	However, since $\lambda {\Upsilon(b\mid C)}^{-1}b\in \lambda C$, the definition of projection yields  
	\begin{equation}\label{eq:proj}
		\left\langle \Pi_{\lambda C}(b)-\lambda {\Upsilon(b\mid C)}^{-1}b,b-\Pi_{\lambda C}(b) \right\rangle\geqslant0.
	\end{equation}
	This contradiction shows the first inequality.
	
	If $b-\Pi_{\lambda C}(b)=0$, then $b\in\lambda C$, which contradicts  $\lambda<\Upsilon(b\mid C)$. Therefore, the second inequality holds.
	
	According to $0\in C$ and the projection property, we have $$\left\langle\Pi_{\lambda C}(b),  b-\Pi_{\lambda C}(b)\right\rangle\geqslant0.$$ If $\left\langle\Pi_{\lambda C}(b),  b-\Pi_{\lambda C}(b)\right\rangle=0$,  then \[
	\begin{aligned}
		&\left\langle \Pi_{\lambda C}(b)-\lambda {\Upsilon(b\mid C)}^{-1}b,b-\Pi_{\lambda C}(b) \right\rangle\\
		=&-\lambda {\Upsilon(b\mid C)}^{-1}\|b-\Pi_{\lambda C}(b)\|^2<0,
	\end{aligned}
	\]which contradicts  the property of projection \eqref{eq:proj}. This completes the proof. 
\end{proof}
The following theorem gives bounds for the derivative of $\psi(\lambda)$ at its differentiable points.
\begin{theorem}\label{th:bound}
	If $\lambda_0\in(0,\Upsilon(b\mid C)) $ is a differentiable point, we have \begin{equation}\label{eq:bounds}
		\psi'(\lambda_0)\in\left[\frac{\left\langle \Pi_{\lambda_0 C}(b),b-\Pi_{\lambda_0 C}(b) \right\rangle^2 }{\lambda_0 \|\Pi_{\lambda_0 C}(b)\|\|b-\Pi_{\lambda_0 C}(b)\|^2},\frac{1}{\lambda_0}\|\Pi_{\lambda_0 C}(b)\|\right] .	
	\end{equation}	
\end{theorem}
\begin{proof}	
	To obtain the bounds of the derivative $\psi'(\lambda_0)$, we construct two functions $\psi_1(\lambda)$, $\psi_2(\lambda)$  s.t. $ \psi_1(\lambda)\leqslant\psi(\lambda)\leqslant\psi_2(\lambda)$ on $(0,\lambda_0)$ and $\psi_1(\lambda_0)=\psi(\lambda_0)=\psi_2(\lambda_0)$ to estimate the bounds of $\psi'(\lambda)$.
	
	We denote $y_0:=\Pi_{\lambda_0 C}(b)$ for convenience. To construct $\psi_1$ and $\psi_2$, we consider the two programs where $\lambda\in(0,\lambda_0]$ is a parameter.
	\begin{equation}\label{eq:bounds_prog1}
		\begin{aligned}
			\min _{z\in\bR^p}\text { or } \max _{z\in\bR^p}  &\quad  \|z\| \\
			\text { s.t. } & \exists \text { closed convex set } 0 \in S \subseteq \mathbb{R}^p, \\
			& \Pi_{\lambda_0 S}(b)={y_0}, \Pi_{\lambda S}(b)=z.
		\end{aligned}
	\end{equation}
	Note that $\Pi_{\lambda C}(b)$ is a feasible point for both minimization and maximization. The optimal values of these two programs could be used as the bounds of $\psi(\lambda)$.
	
	The program \eqref{eq:bounds_prog1} is equivalent to the following program.
	\begin{equation}\label{eq:bounds_prog2}
		\begin{aligned}
			\min _{z\in\bR^p} \text { or } \max _{z\in\bR^p}&\quad  \|z\| \\
			\text { s.t. } & \langle z-\frac{\lambda}{\lambda_0} {y_0}, b-{y_0}\rangle \leqslant 0,\\&  \langle z-\frac{\lambda}{\lambda_0}{y_0}, z-b\rangle \leqslant 0, \\
			& \langle z, z-b\rangle \leqslant 0.
		\end{aligned}
	\end{equation}
	To show the equivalence between \eqref{eq:bounds_prog1} and \eqref{eq:bounds_prog2}, we only need to show that the feasible sets are the same.
	If $z$ is a feasible point of \eqref{eq:bounds_prog2}, it is also a feasible point of \eqref{eq:bounds_prog1} with closed convex $S:=\operatorname{conv}\left\lbrace 0,\frac{1}{\lambda_0}y_0,\frac{1}{\lambda}z\right\rbrace $. For a feasible point $z$ of \eqref{eq:bounds_prog1}, three inequalities in \eqref{eq:bounds_prog2} hold because of the two projections in \eqref{eq:bounds_prog1}. 
	
	For the minimization in \eqref{eq:bounds_prog2}, we consider the following relaxation,
	\begin{equation}\label{eq:bounds_min}
		\min_{z\in\bR^p} \|z\|\quad\text{ s.t. }(\|z\|-\frac{\lambda}{\lambda_0}\|y_0\|)(\|z\|-\|y_0\|)\leqslant0.
	\end{equation}
	Since $\lambda\leqslant\lambda_0$, $\frac{\lambda}{\lambda_0}\|y_0\|$ is the minimal value of \eqref{eq:bounds_min} (actually $z_1=\frac{\lambda}{\lambda_0}y_0$ is a minimal point of \eqref{eq:bounds_prog2}). Therefore, we have 
	\begin{equation}\label{eq:lowbd}
		\psi_1(\lambda):=\frac{\lambda}{\lambda_0}\|y_0\|\leqslant \|\Pi_{\lambda C}(b)\|=\psi(\lambda) \text{ for any } \lambda\in(0,\lambda_0].
	\end{equation}
	
	For maximization in \eqref{eq:bounds_prog2}, we consider the relaxation below,
	\begin{equation}\label{eq:bound_max}
		\begin{aligned}
			\max_{z\in\bR^p}\quad & \|z\|\\
			\text{ s.t. }&\langle z-\frac{\lambda}{\lambda_0} {y_0}, b-{y_0}\rangle \leqslant 0, \\&\langle z-\frac{\lambda}{\lambda_0}{y_0}, z-y_0\rangle \leqslant 0.
		\end{aligned}
	\end{equation}
	The convex maximization \eqref{eq:bound_max} can be equivalently reformulated as the following program
	\begin{equation}\label{eq:bound_max_beta}
		\max_{\beta\in[\frac{\lambda}{\lambda_0},1]} \nu(\beta),
	\end{equation}
	where 
	\begin{equation}\label{eq:bound_max_beta2}
		\begin{aligned}
			\nu(\beta)=&	\max_{z\in\bR^p}\quad  \|z\|\\
			\text{ s.t. }&\langle z-\frac{\lambda}{\lambda_0} {y_0}, b-{y_0}\rangle \leqslant 0, \\ & \langle z-\frac{\lambda}{\lambda_0}{y_0}, z-\beta y_0\rangle = 0.
		\end{aligned}
	\end{equation}
	It is obvious that $\nu(\frac{\lambda}{\lambda_0})=\frac{\lambda}{\lambda_0}\|y_0\|$. For $\beta\in\left( \frac{\lambda}{\lambda_0},1 \right]$, we can also give an explicit solution as follows.
	
	In the program \eqref{eq:bound_max_beta2}, with the equality constraint, the objective $\|z\|$ can be transferred as $\|z-\frac{\lambda}{\lambda_0}y_0\|$ for \[\|z-\frac{\lambda}{\lambda_0}y_0\|^2=\frac{\beta\lambda_0-\lambda}{\beta\lambda_0+\lambda}\|z\|^2+\left( \frac{\lambda^2}{\lambda_0^2}-\frac{2\beta\lambda^2}{(\beta\lambda_0+\lambda)\lambda_0} \right) \|y_0\|^2,\]
	which could be derived from the following identity
	\[ \|z-\frac{\lambda}{\lambda_0}y_0\|^2= \|z-\frac{\lambda}{\lambda_0}y_0\|^2-\frac{2\lambda}{\lambda+\beta\lambda_0}\langle z-\frac{\lambda}{\lambda_0}{y_0}, z-\beta y_0\rangle.\]  
	We claim that the maximizer of \eqref{eq:bound_max_beta2} is \[z_2(\beta)=\beta y_0-\frac{\beta\lambda_0-\lambda}{\lambda_0}\frac{\left\langle y_0,b-y_0 \right\rangle }{\|b-y_0\|^2}(b-y_0).\]
	The feasibility of $z_2(\beta)$ can be verified. To verify the optimality, we first combine the two constraints in \eqref{eq:bound_max_beta2} as follows
	\[\left\langle  z-\frac{\lambda}{\lambda_0} {y_0},  z-\frac{\lambda}{\lambda_0} {y_0}-\frac{\beta\lambda_0-\lambda}{\lambda_0}y_0+\frac{\beta\lambda_0-\lambda}{\lambda_0}\frac{\left\langle y_0,b-y_0 \right\rangle }{\|b-y_0\|^2}(b-{y_0})\right\rangle \leqslant0.\]
	Then the optimality is followed by
	\[\begin{aligned}
		\left\|z-\frac{\lambda}{\lambda_0} {y_0}\right\|^2\leqslant&\left\langle z-\frac{\lambda}{\lambda_0} {y_0},\frac{\beta\lambda_0-\lambda}{\lambda_0}y_0-\frac{\beta\lambda_0-\lambda}{\lambda_0}\frac{\left\langle y_0,b-y_0 \right\rangle }{\|b-y_0\|^2}(b-{y_0})\right\rangle \\
		\leqslant&\left\|z-\frac{\lambda}{\lambda_0} {y_0}\right\|\left\|\frac{\beta\lambda_0-\lambda}{\lambda_0}y_0-\frac{\beta\lambda_0-\lambda}{\lambda_0}\frac{\left\langle y_0,b-y_0 \right\rangle }{\|b-y_0\|^2}(b-{y_0}) \right\|. 
	\end{aligned}
	\]
	It is obvious that $\|z-\frac{\lambda}{\lambda_0}y_0\|$ takes its maximal value at $z_2(\beta)$, and the corresponding maximal value in \eqref{eq:bound_max_beta2} is
	\[\|z_2(\beta)\|=\sqrt{\beta^2\|y_0\|^2+\frac{\lambda^2-\beta^2\lambda_0^2}{\lambda_0^2}\frac{\left\langle y_0,b-y_0 \right\rangle^2 }{\|b-y_0\|^2}}.\]
	With $\|y_0\|^2\geqslant \frac{\left\langle y_0,b-y_0 \right\rangle^2 }{\|b-y_0\|^2}$, we know that the maximal value of \eqref{eq:bound_max_beta} is taken at $\beta=1$. Similarly, $z_2(1)$ is actually feasible for \eqref{eq:bounds_prog2}.

	The maximal value $\|z_2(1)\|$ is the upper bound function $\psi_2(\lambda)$ that we construct: \begin{equation}\label{eq:uppbd}
		\begin{aligned}
			\psi_2(\lambda)&:=\sqrt{\|y_0\|^2+\frac{\lambda^2-\lambda_0^2}{\lambda_0^2}\frac{\left\langle y_0,b-y_0 \right\rangle^2 }{\|b-y_0\|^2}}\\&\geqslant \|\Pi_{\lambda C}(b)\|=\psi(\lambda) \text{ for any } \lambda\in(0,\lambda_0]. 
		\end{aligned}
	\end{equation}
	
	Moreover, we have $\psi_1(\lambda_0)=\psi(\lambda_0)=\psi_2(\lambda_0)$.
	
	With the bound function $\psi_1(\lambda)$ and $\psi_2(\lambda)$, we can estimate the bounds of the derivative $\psi'(\lambda_0)$:
	\[
	\psi'(\lambda_0)=\lim_{\lambda\uparrow \lambda_0}\frac{\psi(\lambda)-\psi(\lambda_0)}{\lambda-\lambda_0}\in[\psi_2'(\lambda_0),\psi_1'(\lambda_0)].
	\]
	The explicit form of the above interval is 
	\[
	[\psi_2'(\lambda_0),\psi_1'(\lambda_0)]=\left[\frac{\left\langle \Pi_{\lambda_0 C}(b),b-\Pi_{\lambda_0 C}(b) \right\rangle^2 }{\lambda_0 \|\Pi_{\lambda_0 C}(b)\|\|b-\Pi_{\lambda_0 C}(b)\|^2},\frac{1}{\lambda_0}\|\Pi_{\lambda_0 C}(b)\|\right]. 
	\] The fraction in the lower bound is well defined according to Lemma \ref{lm:neq0}. This completes the proof.
\end{proof}
\begin{theorem}
	For any $\bar \lambda\in(0,\Upsilon(b\mid C))$, the Clarke generalized Jacobian $\partial\psi(\bar\lambda)$ is a subset of the real positive numbers $\bR_{++}$.
\end{theorem}
\begin{proof}
	If $\bar{\lambda}$ is a differentiable point, we have $\partial\psi(\bar{\lambda})=\{\psi'(\bar{\lambda})\}$, whose unique element is a positive number according to Lemma \ref{lm:neq0} and Theorem \ref{th:bound}.
	
	For nondifferentiable $\bar{\lambda}\in(0,\Upsilon(b\mid C))$, we can find a closed neighborhood $\mathcal{B}$ of $\bar{\lambda}$ such that $0<m_\psi\leqslant\psi_2'(\lambda)\leqslant\psi'(\lambda)\leqslant\psi_1'(\lambda)\leqslant M_\psi<+\infty$ for any differentiable point $\lambda\in\mathcal{B}$. The existence of such uniform bounds $m_\psi$ and $M_\psi$ is due to the continuity of $\psi_1'(\lambda)$ and $\psi_2'(\lambda)$. By the definition of the Clarke generalized Jacobian, any element $v\in\partial\psi(\bar{\lambda})$ satisfies $v\in[m_\psi, M_\psi]\subset\bR_{++}$. This completes the proof.
\end{proof}
\begin{corollary}
	For any $\lambda\in(0,\|\cA^*b\|_2)$, the Clarke generalized Jacobian $\partial\varphi(\lambda)\subset\bR_{++}$.
\end{corollary}
\begin{proof}
	The value function $\varphi(\lambda)$ is an example of \eqref{eq:pjsq}. And $\|\cA^*b\|_2=\Upsilon(b\mid Q)$. This completes the proof.
\end{proof}
\subsection{A secant method for the univariate equation \eqref{eq:une}}
This subsection develops a secant method for the univariate equation $\varphi(\lambda) = \varrho$ and analyzes its convergence. The analysis includes studying the semismoothness of $\varphi(\lambda)$.

Given two initial points $\lambda^{-1}$ and $\lambda^0$, the $k-$th iteration of the classic secant method for the univariate equation $\varphi(\lambda)=\varrho$ is as follows:

\begin{equation}\label{eq:sec1}
	\lambda^{k+1}
	=
	\lambda^k
	-
	\frac{\lambda^k-\lambda^{k-1}}
	{\varphi(\lambda^k)-\varphi(\lambda^{k-1})}
	\bigl(\varphi(\lambda^k)-\varrho\bigr),
	\qquad k=0,1,\ldots.
\end{equation}
Like Newton methods, secant methods achieve fast convergence when the Clarke generalized Jacobian is nonsingular and the function is semismooth at the optimal point. The nonsingularity condition was examined earlier; now we analyze the semismoothness of $\varphi(\lambda)$.
\begin{proposition}\label{prop:ss}
For $Q=\{z\in\bR^p\mid \|\cA^*z\|_2\leqslant1\}$, the projection function $\Pi_Q(\cdot)$ is $\gamma-$order semismooth and the value function $\varphi(\lambda)=\|\lambda\Pi_Q(\lambda^{-1}b)\|$ is also $\gamma-$order semismooth on $(0,+\infty)$ with some $\gamma>0$.
\end{proposition}
\begin{proof}
With its linear matrix inequality representation, we know that $Q$ is a semialgebraic set. Then its projection $\Pi_Q(\cdot)$ is a semialgebraic function. From \cite[Remark 4]{Bolte2009}, we know that $\Pi_Q(\cdot)$ is $\gamma-$order semismooth in the sense of \eqref{eq:def-gammass2}. The $\gamma-$order semismoothness of the value function $\varphi(\lambda)$ is due to the composition of $\gamma-$order semismooth functions \cite[Theorem 19]{Fischer1997}. This completes the proof.
\end{proof}

Proposition \ref{prop:ss} establishes the existence of some $\gamma>0$. The following theorem provides sufficient conditions for the stronger result $\gamma=1$,  corresponding to the strong semismoothness of $\varphi(\lambda)$. To facilitate further discussion, we reformulate the dual of the nuclear norm regularized problem \eqref{eq:nnrls-d} and the primal problem \eqref{eq:nnrls} as the following conic programs:
\begin{equation}\label{eq:nnrls-dc}
\min _{y \in \mathbb{R}^{p}}\frac{1}{2}\|y\|^2-b^Ty\quad\text{ s.t. } (\lambda,\cA^*y)\in\operatorname{epi}\|\cdot\|_2 ;
\end{equation}
\begin{equation}\label{eq:nnls-pc}
\min _{(s, X) \in \mathbb{R} \times \mathbb{R}^{m \times n}}\lambda s+\frac{1}{2}\|\mathcal{A} X-b\|^2 \quad\text{ s.t. } (s, X) \in \operatorname{epi}\|\cdot\|_*.
\end{equation}
\begin{proposition}\label{prop:sss}
Let $y(\lambda^*)$ be the optimal solution of \eqref{eq:nnrls-dc} with $\lambda=\lambda^*$. If $y(\lambda^*)$ is constraint nondegenerate to \eqref{eq:nnrls-dc}, i.e.,
\begin{equation}\label{eq:cn_nnlsdc}
	(0,\cA^*\bR^p)+\operatorname{lin}\left( T_{\epi\|\cdot\|_2}(\lambda^*,\cA^*y(\lambda^*)) \right)=\bR\times\bR^{m\times n},
\end{equation}
then the value function $\varphi(\cdot)$ is strongly semismooth at $\lambda^*$.
\end{proposition}
\begin{proof}
The primal-dual solution $(y(\lambda^*),s(\lambda^*),X(\lambda^*))$ of \eqref{eq:nnrls-dc} and \eqref{eq:nnls-pc}  with $\lambda=\lambda^*$ satisfies the nonsmooth equation
\begin{equation}\label{eq:nnls-kktc}
	G(\lambda^*,y, s, X)=\left[\begin{array}{c}
		y-b+\mathcal{A} X \\
		\left(\lambda^*, \mathcal{A}^* y \right)-\Pi_{\operatorname{epi} \|\cdot \|_2}\left(\left(\lambda^*,\mathcal{A}^*y\right)-(s,-X)\right)
	\end{array}\right]=0,
\end{equation}
which is equivalent to the KKT condition \eqref{eq:nnrls-kkt}. Given the strong convexity of the objective in \eqref{eq:nnrls-dc} and the constraint nondegeneracy \eqref{eq:cn_nnlsdc}, we know that any element of $\partial_{y,S,X} G(\lambda^*,y(\lambda^*),s(\lambda^*),X(\lambda^*))$ is nonsingular \cite[Theorem 4.1]{Guo2015}. According to \cite[Theorem 3]{Meng2005}, there exists an open interval $\mathcal{N}$ containing $\lambda^*$ and a Lipschitz continuous function $(y(\lambda),s(\lambda),X(\lambda))$ on $\mathcal{N}$ s.t. $G(\lambda,y(\lambda),s(\lambda),X(\lambda))=0$. Additionally, the projection $\Pi_{\operatorname{epi} \|\cdot \|_2}(\cdot)$ is strongly semismooth \cite[Theorem 4, Theorem 2]{Ding2014}. Utilizing Proposition \ref{prop:in-fun-th} and \cite[Theorem 3]{Meng2005}, we conclude that $(y(\lambda), s(\lambda), X(\lambda))$ exhibits strong semismoothness at $\lambda^*$. In addition, given $y(\lambda) = \Pi_{\lambda Q}(b) \neq 0$ and $\varphi(\lambda) = \|y(\lambda)\|$, the strong semismoothness of $\varphi(\cdot)$ at $\lambda^*$ follows from the composition of strongly semismooth functions. This completes the proof.
\end{proof}

\begin{remark}
The set $Q$ is an affine slice of the PSD cone, i.e., the intersection of the PSD cone with an affine subspace. A key property of such sets, recently proven in \cite{Chen2025}, is that their projection operator can have an arbitrarily low degree of semismoothness (meaning $\gamma$ in Proposition \ref{prop:ss} can be arbitrarily close to 0). This fact underscores the importance of the conditions imposed in Proposition \ref{prop:sss}, as they are required to preclude this pathological behavior and secure a fast convergence rate.
\end{remark}
We next state the local convergence properties of the exact secant
iteration \eqref{eq:sec1}.

\begin{proposition}\label{prop:sec_converge}
Let $\lambda^*$ be a solution of $\varphi(\lambda)=\varrho$, and suppose
that $\varphi$ is $\gamma$-order semismooth at $\lambda^*$ for some
$\gamma>0$. Let
\[
d^-:=-\varphi'(\lambda^*;-1)
\qquad\text{and}\qquad
d^+:=\varphi'(\lambda^*;1)
\]
be the two lateral derivatives of $\varphi$ at $\lambda^*$.

\begin{enumerate}[(i)]
	\item Suppose that $d^-$ and $d^+$ are both positive or both
	negative. Then there exists a neighborhood $\mathcal U$ of
	$\lambda^*$ such that, for any distinct
	$\lambda^{-1},\lambda^0\in\mathcal U$, the secant iteration
	\eqref{eq:sec1} is well defined, remains in a neighborhood of
	$\lambda^*$, and converges to $\lambda^*$. Moreover,
	\begin{equation}\label{eq:exact-three-step}
		|\lambda^{k+3}-\lambda^*|
		=
		O\bigl(|\lambda^k-\lambda^*|^{1+\gamma}\bigr).
	\end{equation}
	
	\item Suppose, in addition, that
	\[
	\partial\varphi(\lambda^*)=\{v^*\},
	\qquad v^*\neq0.
	\]
	Let
	\[
	e_k:=\lambda^k-\lambda^*,
	\qquad
	p_\gamma:=\frac{1+\sqrt{1+4\gamma}}{2},
	\]
	and, whenever the denominator is nonzero, define
	\[
	c_k
	:=
	\frac{|e_k|}
	{|e_{k-1}|\,|e_{k-2}|^\gamma},
	\qquad k\geq1.
	\]
	Then $\{\lambda^k\}$ converges to $\lambda^*$ Q-superlinearly.
	Furthermore, one of the following two alternatives holds:
	\begin{enumerate}[(a)]
		\item $\{\lambda^k\}$ converges to $\lambda^*$ with Q-order
		at least $p_\gamma$;
		
		\item $\{\lambda^k\}$ converges to $\lambda^*$ with R-order
		at least $p_\gamma$, and, for every $\underline C>0$, there
		exists a subsequence $\{c_{i_j}\}$ satisfying
		\[
		c_{i_j}<\underline C\,i_j^{-i_j}.
		\]
	\end{enumerate}
\end{enumerate}
\end{proposition}

\begin{proof}
Part (i) follows from the $\gamma$-order version of the local secant
analysis in \cite[Theorem 3.2]{Potra1998}; see also
\cite[Proposition 5.1]{Li2024}.

For part (ii), the singleton assumption gives
\[
d^-=d^+=v^*.
\]
The $\gamma$-order semismoothness estimates for univariate divided
differences imply that, for all sufficiently large $k$,
\begin{equation}\label{eq:exact-singleton-recursion}
	|e_{k+1}|
	\leq
	C|e_k|
	\bigl(|e_k|^\gamma+|e_{k-1}|^\gamma\bigr)
\end{equation}
for some $C>0$. Consequently, the convergence is Q-superlinear and,
eventually,
\[
|e_{k+1}|
\leq
\widetilde C|e_k|\,|e_{k-1}|^\gamma.
\]
The asserted Q-order/R-order alternatives follow from
Proposition~\ref{prop:1ss},
\cite[Lemma 2.2, Lemma 2.3]{Potra1998},
\cite[Proposition 5.2]{Li2024}, and
\cite[Theorem 2.2, Corollary 3.1]{Potra1989}.
\end{proof}

Proposition~\ref{prop:sec_converge} ensures only local convergence. For global
convergence, we adopt a stabilized secant method; see, for example,
\cite[Algorithm 5.2]{Li2024}. The remaining numerical aspects are addressed
in Section~\ref{sec:numeric-issue}.

The preceding local convergence analysis assumes exact evaluations of
the value function. For completeness, we next record a perturbation
result for approximate evaluations. Its purpose is to identify
sufficient local accuracy conditions under which the convergence rates
of the exact secant iteration are retained.

 At the
$k$th outer iteration, let
\begin{equation}\label{eq:inexact-value}
\widehat\varphi_k
=
\varphi(\lambda^k)+\xi_k,
\end{equation}
where $\xi_k$ denotes the error in the computed value function. The corresponding
inexact secant iteration is
\begin{equation}\label{eq:inexact-secant}
\lambda^{k+1}
=
\lambda^k
-
\frac{\lambda^k-\lambda^{k-1}}
{\widehat\varphi_k-\widehat\varphi_{k-1}}
\bigl(\widehat\varphi_k-\varrho\bigr).
\end{equation}

The following condition controls both the error in the function value and the
error in the secant slope:
\begin{equation}\label{eq:inexact-condition}
\frac{|\xi_k|}{|e_k|}
+
\frac{|\xi_k-\xi_{k-1}|}
{|\lambda^k-\lambda^{k-1}|}
=
O\bigl(|e_k|^\gamma+|e_{k-1}|^\gamma\bigr),
\qquad
e_k:=\lambda^k-\lambda^*.
\end{equation}

\begin{proposition}\label{prop:inexact-sec-converge}
Let $\lambda^*$ be a solution of $\varphi(\lambda)=\varrho$, and suppose
that $\varphi$ is $\gamma$-order semismooth at $\lambda^*$ for some
$\gamma>0$. Assume that the inexact function values satisfy
\eqref{eq:inexact-condition}.

\begin{enumerate}[(i)]
	\item If the lateral derivatives satisfy $d^-,d^+>0$, then there
	exists a neighborhood $\mathcal U$ of $\lambda^*$ such that, for
	any sufficiently close distinct initial points
	$\lambda^{-1},\lambda^0\in\mathcal U$, the inexact secant iteration
	\eqref{eq:inexact-secant} is well defined and converges to
	$\lambda^*$. Moreover,
	\begin{equation}\label{eq:inexact-three-step}
		|\lambda^{k+3}-\lambda^*|
		=
		O\bigl(|\lambda^k-\lambda^*|^{1+\gamma}\bigr).
	\end{equation}
	
	\item If, in addition,
	\[
	\partial\varphi(\lambda^*)=\{v^*\},
	\qquad v^*\neq0,
	\]
	then $\{\lambda^k\}$ converges to $\lambda^*$ Q-superlinearly.
	Moreover, the Q-order/R-order alternatives in
	Proposition~\ref{prop:sec_converge}(ii) remain valid, with
	$
	p_\gamma=\frac{1+\sqrt{1+4\gamma}}{2}$.
	
\end{enumerate}
\end{proposition}

\begin{proof}
Prior to termination, define
\[
s_k
:=
\frac{\varphi(\lambda^k)-\varphi(\lambda^{k-1})}
{\lambda^k-\lambda^{k-1}},
\qquad
t_k
:=
\frac{\varphi(\lambda^k)-\varrho}{e_k},
\]
and
\[
q_k
:=
\frac{\xi_k-\xi_{k-1}}
{\lambda^k-\lambda^{k-1}}.
\]
Then
\[
\widehat\varphi_k-\widehat\varphi_{k-1}
=
(s_k+q_k)(\lambda^k-\lambda^{k-1})
\]
and
\[
\widehat\varphi_k-\varrho
=
t_ke_k+\xi_k.
\]
Substituting these identities into \eqref{eq:inexact-secant} yields
\begin{equation}\label{eq:inexact-error-identity}
	\begin{aligned}
		e_{k+1}
		&=
		\frac{(s_k-t_k+q_k)e_k-\xi_k}
		{s_k+q_k}
		\\
		&=
		\frac{
			s_k-t_k+q_k-\xi_k/e_k
		}{
			s_k+q_k
		}e_k.
	\end{aligned}
\end{equation}

Since $d^-,d^+>0$, the exact divided differences $s_k$ are uniformly
bounded away from zero whenever $\lambda^k$ and $\lambda^{k-1}$ are
sufficiently close to $\lambda^*$. Condition
\eqref{eq:inexact-condition} gives
\[
q_k
=
O\bigl(|e_k|^\gamma+|e_{k-1}|^\gamma\bigr),
\qquad
\frac{\xi_k}{e_k}
=
O\bigl(|e_k|^\gamma+|e_{k-1}|^\gamma\bigr).
\]
Consequently, $s_k+q_k$ is also uniformly bounded away from zero.

To make the connection with the proof of
\cite[Theorem~3.2]{Potra1998} explicit, suppose first that
$\lambda^{k-1}$ and $\lambda^k$ lie on the same side of $\lambda^*$.
If $e_{k-1},e_k>0$, by the same argument as in \cite[Lemma~2.3]{Potra1998}, using the
$\gamma$-order semismooth remainder, the $\gamma$-order semismooth
divided-difference estimates give
\[
s_k
=
d^+
+
O\bigl(|e_k|^\gamma+|e_{k-1}|^\gamma\bigr),
\qquad
t_k
=
d^+
+
O(|e_k|^\gamma).
\]
The analogous estimates hold with $d^-$ when $e_{k-1},e_k<0$.
Together with \eqref{eq:inexact-condition} and
\eqref{eq:inexact-error-identity}, this yields
\[
|e_{k+1}|
\leq
C|e_k|
\bigl(|e_k|^\gamma+|e_{k-1}|^\gamma\bigr).
\]
This is the $\gamma$-order inexact counterpart of cases (i) and (ii)
in the proof of \cite[Theorem~3.2]{Potra1998}.

Now suppose that the two iterates lie on opposite sides of the
solution, say $e_{k-1}<0<e_k$. Splitting the divided difference at
$\lambda^*$, as in the proof of
\cite[Theorem~3.2]{Potra1998}, gives
\[
\begin{aligned}
	s_k
	&=
	\frac{\varphi(\lambda^k)-\varphi(\lambda^*)}
	{\lambda^k-\lambda^{k-1}}
	+
	\frac{\varphi(\lambda^*)-\varphi(\lambda^{k-1})}
	{\lambda^k-\lambda^{k-1}}
	\\
	&=
	\frac{e_k}{e_k-e_{k-1}}\,d^+
	+
	\frac{-e_{k-1}}{e_k-e_{k-1}}\,d^-
	+
	O\bigl(|e_k|^\gamma+|e_{k-1}|^\gamma\bigr).
\end{aligned}
\]
Since
\[
t_k=d^++O(|e_k|^\gamma),
\]
condition \eqref{eq:inexact-condition} implies
\begin{equation}\label{eq:inex-cross}
	s_k+q_k
	=
	\frac{e_k}{e_k-e_{k-1}}\,d^+
	+
	\frac{-e_{k-1}}{e_k-e_{k-1}}\,d^-
	+
	O\bigl(|e_k|^\gamma+|e_{k-1}|^\gamma\bigr),
\end{equation}
\begin{equation}\label{eq:inex-tmp}
	s_k-t_k+q_k-\frac{\xi_k}{e_k}
	=
	\frac{-e_{k-1}}{e_k-e_{k-1}}\,(d^--d^+)
	+
	O\bigl(|e_k|^\gamma+|e_{k-1}|^\gamma\bigr).
\end{equation}

The case $e_k<0<e_{k-1}$ is symmetric. Equation \eqref{eq:inex-cross} is the
$\gamma$-order inexact counterpart of the crossing-point
divided-difference decomposition used immediately before
\cite[Eq.~(3.5)]{Potra1998}, while
\eqref{eq:inexact-error-identity} is the inexact counterpart of
\cite[Eq.~(3.4)]{Potra1998}.

Because
\(
0<
\frac{-e_{k-1}}{e_k-e_{k-1}}
<1
\)
and $d^-,d^+>0$, for all sufficiently large $k$ we have
\[
s_k+q_k
\geq
\min\{d^-,d^+\}
-
C\bigl(|e_k|^\gamma+|e_{k-1}|^\gamma\bigr)>0,
\]
and
\[
\left|
s_k-t_k+q_k-\frac{\xi_k}{e_k}
\right|
\leq
|d^+-d^-|
+
C\bigl(|e_k|^\gamma+|e_{k-1}|^\gamma\bigr).
\]
Consequently,
\begin{equation}\label{eq:inex3.5}
	|e_{k+1}|
	\leq
	\frac{
		|d^+-d^-|
		+
		C\bigl(|e_k|^\gamma+|e_{k-1}|^\gamma\bigr)
	}{
		\min\{d^-,d^+\}
		-
		C\bigl(|e_k|^\gamma+|e_{k-1}|^\gamma\bigr)
	}
	|e_k|.
\end{equation}
This is the inexact counterpart of
\cite[Eq.~(3.5)]{Potra1998}.

To make the sign-pattern argument following
\cite[Eq.~(3.5)]{Potra1998} explicit, assume without loss of
generality that
$
0<d^-\leqslant d^+.
$
The case $0<d^+\leqslant d^-$ is symmetric. Since
$
0<
\frac{-e_{k-1}}{e_k-e_{k-1}}
<1,
$
the second expansion above implies
\[
\begin{aligned}
	d^- - d^+
	-C\bigl(|e_k|^\gamma+|e_{k-1}|^\gamma\bigr)
	&\leq
	s_k-t_k+q_k-\frac{\xi_k}{e_k}
	\\
	&\leq
	C\bigl(|e_k|^\gamma+|e_{k-1}|^\gamma\bigr).
\end{aligned}
\]
Since $s_k+q_k>0$ and $e_k>0$, the error identity
\eqref{eq:inexact-error-identity} therefore gives
\begin{equation}\label{eq:inex-twoside}
	\begin{aligned}
		&
		\frac{
			d^- - d^+
			-C\bigl(|e_k|^\gamma+|e_{k-1}|^\gamma\bigr)
		}{
			s_k+q_k
		}\,e_k
		\\
		&\hspace{3em}\leqslant e_{k+1}\leq
		\frac{
			C\bigl(|e_k|^\gamma+|e_{k-1}|^\gamma\bigr)
		}{
			s_k+q_k
		}\,e_k .
	\end{aligned}
\end{equation}
This is precisely the perturbed $\gamma$-order analogue of the
two-sided inequality immediately following
\cite[Eq.~(3.5)]{Potra1998}. Together with the preceding same-side estimate,
\eqref{eq:inex3.5} and \eqref{eq:inex-twoside} imply, after shrinking
the neighborhood if necessary, that
$
|e_{j+1}|\leqslant M|e_j|
$
for all subsequent local iterates and some $M>0$. These estimates are
the $\gamma$-order inexact counterparts of those used in the proof of
\cite[Theorem~3.2]{Potra1998}. Since the additional perturbation terms
have the same order as the semismooth remainders, the sign-pattern
argument following \cite[Eq.~(3.5)]{Potra1998}, in particular the
argument leading to \cite[Eqs.~(3.6)--(3.8)]{Potra1998}, carries over
and yields
$
|e_{k+3}|
=
O\bigl(|e_k|^{1+\gamma}\bigr).
$
This proves \eqref{eq:inexact-three-step}.

Under the singleton assumption, $d^-=d^+=v^*$, and the
$\gamma$-order semismooth divided-difference estimates give
\[
s_k
=
v^*
+
O\bigl(|e_k|^\gamma+|e_{k-1}|^\gamma\bigr),
\qquad
t_k
=
v^*+O(|e_k|^\gamma).
\]
Together with \eqref{eq:inexact-condition}, substitution into
\eqref{eq:inexact-error-identity} gives
\begin{equation}\label{eq:inexact-singleton-recursion}
	|e_{k+1}|
	\leq
	C|e_k|
	\bigl(|e_k|^\gamma+|e_{k-1}|^\gamma\bigr)
\end{equation}
for all sufficiently large $k$ and some $C>0$. Thus the convergence
is Q-superlinear. In particular, for all sufficiently large $k$,
\[
|e_{k+1}|
\leq
\widetilde C|e_k|\,|e_{k-1}|^\gamma.
\]
The Q-order/R-order alternatives follow from the same sequence
arguments used in Proposition~\ref{prop:sec_converge}(ii). This completes the proof.
\end{proof}
\section{A proximal generation method for \eqref{eq:nnrls}}
In our level set method for \eqref{eq:lscnnm}, we need to solve a sequence of regularized subproblems \eqref{eq:nnrls} with varying $\lambda$. Our key innovation is a proximal generation method that warm-starts each subproblem from the previous solution $X^{k-1}$ and exploits the low-rank structure of the solutions to reduce computational cost. The term generation refers to the role of each proximal gradient step in generating a new low-rank factorization for the next alternating minimization stage.

For large-scale problems, the following factorization form is much more efficient than operating on the full matrix:
\begin{equation}\label{eq:lr-nnrls}
	\min _{L \in \mathbb{R}^{m\times r}, R\in \bR^{n\times r}}g_\lambda(L,R):=\frac{\lambda}{2}(\|L\|^2+\|R\|^2) +\frac{1}{2}\|\cA (LR^T)-b\|^2 .
\end{equation}
The equivalence between \eqref{eq:nnrls} and \eqref{eq:lr-nnrls} is summarized in the following proposition, which can be proved by a similar approach in \cite[Lemma 5.1]{Recht2010}.

\begin{proposition}\label{prop:nn-lr-eq}
	Let $X^*$ be a solution to the original convex problem \eqref{eq:nnrls} with $\operatorname{rank}(X^*) \leqslant r$. The optimal value of \eqref{eq:lr-nnrls} is the same as \eqref{eq:nnrls}. Moreover, for any factorization $X=LR^T$, we have
	$f_\lambda(X)\leqslant g_\lambda(L,R)$.
	If $X$ has an SVD $X=U_X\Sigma_XV_X^T$ and $L=U_X\sqrt{\Sigma_X}$, $R=V_X\sqrt{\Sigma_X}$, then
	$f_\lambda(X)= g_\lambda(L,R)$.
\end{proposition}

Our proximal generation method evaluates the quality of a candidate solution $X^k$ using the (relative) subgradient residual (r)SGR. Building on the SGR measure from \cite{Toh2010} and its discussion in \cite[Section 4.2.1]{Du2015}, we define these residuals as follows:
\begin{subequations}\label{eq:pgr-sgr-rsgr}
	\begin{align}
		\operatorname{PG}_\lambda(X^k):=&\operatorname{prox}_{(\lambda/L_{\cA^*\cA})\|\cdot\|_*}\left( X^k-\frac{1}{L_{\cA^*\cA}}\cA^*\left(\cA X^k-b \right)  \right),\label{eq:pg}\\
		\operatorname{SGR}_\lambda(X^k):=&L_{\cA^*\cA}(X^k-\operatorname{PG}_\lambda(X^k))+\cA^*\left( \cA(\operatorname{PG}_\lambda(X^k)-X^{k})\right),\label{eq:sgr}\\
		\text{ and }  \operatorname{rSGR}_\lambda(X^k):=&\operatorname{SGR}_\lambda(X^k)/\left(L_{\cA^*\cA}(1+\|\operatorname{PG}_\lambda(X^k)\|) \right)\label{eq:rsgr} ,
	\end{align}
\end{subequations}
where $L_{\cA^*\cA}$ is the Lipschitz constant of the gradient of $\frac{1}{2}\|\cA X-b\|^2$. For completeness and clarity, we formally validate this approach in the following proposition.

\begin{proposition}\label{prop:sgr}
	Let $\{X^k\}$ be a sequence converging to a solution of \eqref{eq:nnrls}. It holds that: 
	
	(i) matrix $X^*$ satisfies $X^*=\operatorname{PG}_\lambda(X^*)$ if and only if $X^*$ is a solution of \eqref{eq:nnrls};
	
	(ii) $\operatorname{SGR}_\lambda(X^k)\in \partial f_\lambda(\operatorname{PG}_\lambda(X^k))$;
	
	(iii) $\lim_{k\to\infty}\|\operatorname{SGR}_\lambda(X^k)\|=0$.
\end{proposition}
\begin{proof}
	Results (i) and (ii) follow directly from the definition of the proximal mapping. We focus on the third conclusion.
	
	Since $X^k$ converges to a solution of \eqref{eq:nnrls}, the continuity of the proximal mapping implies that $\operatorname{PG}_\lambda(X^k)-X^k$ converges to zero. By the definition of $\operatorname{SGR}_\lambda(X^k)$ and the continuity of the mapping $\cA^*\cA$, we conclude that $\lim_{k\to\infty}\|\operatorname{SGR}_\lambda(X^k)\|=0$. This completes the proof.
\end{proof}

In our algorithm, we terminate the iteration when 
$\|\operatorname{rSGR}_\lambda(X^k)\|<\varepsilon$.
We frequently compute a low-rank decomposition of a matrix~$X$ from its singular value decomposition (SVD), $X = U_X \Sigma_X V_X^T$. Defining $L = U_X \sqrt{\Sigma_X}$ and $R = V_X \sqrt{\Sigma_X}$, we denote this decomposition by 
$
\operatorname{LR-SVD}(X) = \bigl(U_X \sqrt{\Sigma_X},\; V_X \sqrt{\Sigma_X}\bigr)$.

\begin{algorithm}[!h]
	\caption{A proximal generation method for \eqref{eq:nnrls}}
	\label{alg:RS}
	\begin{algorithmic}[1]
		\Require An initial solution  $X^{0}\in\bR^{m\times n}$, tolerances $\varepsilon,\varepsilon_{L,R}\geqslant0$, step length of proximal gradient $1/L_{\cA^*\cA}$, and a positive integer $k_{\text{check}}$.
		\State Compute a step of proximal gradient $X^0_+=\operatorname{PG}_\lambda(X^0)$.
		\State		Set $(L^1,R^1)=\operatorname{LR-SVD}(X^0_+)$, $\text{Res}=\|\operatorname{rSGR}_\lambda(X^0)\|$ and $k=1$.
		\While{$\text{Res}>\varepsilon$}
		\State Alt. Min.: $L^{k+1}=\argmin_L g_\lambda(L,R^k)$;$\quad R^{k+1}=\argmin_R g_\lambda(L^{k+1},R)$.	
		\State	Set $k\leftarrow k+1$.
		\If{$\mod(k,k_{\text{check}})=0$ or $\|\nabla g_\lambda(L^k,R^k)\|\leqslant\varepsilon_{L,R}$}
		\State 	Set $X^{k}=L^{k}{R^{k}}^T$. Compute a step of proximal gradient $X^k_+=\operatorname{PG}_\lambda(X^k)$.
		\State	Set $(L^{k+1},R^{k+1})=\operatorname{LR-SVD}(X^k_+)$, $\text{Res}=\|\operatorname{rSGR}_\lambda(X^k)\|$ and $k\leftarrow k+1$.
		\EndIf
		\EndWhile
		
		\State Return $X^*=X^{k-1}_+=L^k{R^k}^T$.
	\end{algorithmic}
\end{algorithm}

Our proximal generation method (Algorithm \ref{alg:RS}) is efficient because it exploits the problem's low-rank structure via the nonconvex formulation \eqref{eq:lr-nnrls}. We solve it by alternating minimization, with each subproblem handled efficiently. Proximal gradient steps provide global convergence and strong warm starts for nonconvex updates. These features make the method well-suited to the nuclear norm regularized least-squares problem \eqref{eq:nnrls}.

Let ${k_i}$ denote the indices at which the proximal gradient step is executed. We have
\[f_\lambda(X_+^{k_{i+1}})\leqslant f_\lambda(X^{k_{i+1}})\leqslant g_\lambda(L^{k_{i+1}},R^{k_{i+1}})\leqslant g_\lambda(L^{k_{i}+1},R^{k_{i}+1})=f_\lambda(X^{k_i}_+).\]
The first inequality follows from \cite[Lemma 2.3]{Beck2009}, the second and last equalities from Proposition \ref{prop:nn-lr-eq}, and the third from alternating minimization. The convergence of our proximal generation method (Algorithm \ref{alg:RS}) is shown as follows.
\begin{proposition}\label{prop:conv-rs}
	Let $\{X^{k_i}\}$ and $\{X^{k_i}_+\}$ be sequences generated by proximal gradient steps in the proximal generation method (Algorithm \ref{alg:RS}) with $\varepsilon=0$, we have, for any $i\geqslant1$, \[f_\lambda(X^{k_i})-f_\lambda(X^*)=O\left( \frac{1}{i}\right) ,\quad \forall X^*\in\argmin f_\lambda(X), \]
	and any limiting point of $\{X_+^{k_i}\}$ is a solution of \eqref{eq:nnrls}. 
\end{proposition}  
\begin{proof}
	The \(O\!\left(\frac{1}{i}\right)\) complexity of the function values follows from \cite[Theorem 4]{Lee2024}. The existence of a limiting point of \(\{X_+^{k_i}\}\) is guaranteed by the coerciveness of \(f_\lambda\). Moreover, since \(X^{k_i}-X_+^{k_i}\to 0\), the corresponding subgradient residual converges to zero. Therefore, by the upper semicontinuity of \(\partial f_\lambda(\cdot)\) and Proposition \ref{prop:sgr}(ii), any limiting point of \(\{X_+^{k_i}\}\) is optimal. This completes the proof.
\end{proof}
In practice, the alternating minimization method is highly effective, typically requiring only 2 to 5 proximal gradient steps.
\begin{remark}
	In Algorithm 2, we use \(\|\operatorname{rSGR}(X)\|\) rather than \(\|X-\operatorname{prox}_{\lambda\|\cdot\|_*}(X-\mathcal{A}^*(\mathcal{A} X-b))\|\) to assess the solution quality, since evaluating the latter requires an additional SVD.
\end{remark}

For $\lambda>0$, denote the solution set of the regularized problem by
\[
\Omega(\lambda):=\argmin_X f_\lambda(X).
\]
Given an approximate inner iterate $W$, define
\[
\widetilde X:=\operatorname{PG}_\lambda(W),
\qquad
g:=\operatorname{SGR}_\lambda(W),
\]
and use
\begin{equation}\label{eq:approx-value-from-pg}
	\widehat\varphi(\lambda;W)
	:=
	\|\mathcal A\widetilde X-b\|
\end{equation}
as the approximate value function. We next derive sufficient inexactness conditions under which the local
convergence rate of the outer secant iteration is preserved.
\begin{proposition}\label{prop:sgr-inexact}
	Let $\mathcal N\subset(0,+\infty)$ be a compact interval containing
	$\lambda^*$. For any $\lambda\in\mathcal N$ and any approximate inner
	iterate $W$, define
	\[
	\widetilde X:=\operatorname{PG}_\lambda(W),
	\qquad
	g:=\operatorname{SGR}_\lambda(W),
	\qquad
	\widehat\varphi(\lambda;W)
	:=
	\|\mathcal A\widetilde X-b\|.
	\]
	Then the following statements hold.
	
	\begin{enumerate}[(i)]
		\item For every $X_\lambda\in\Omega(\lambda)$,
		\begin{equation}\label{eq:sgr-basic-value-bound}
			\bigl|
			\widehat\varphi(\lambda;W)-\varphi(\lambda)
			\bigr|^2
			\leq
			\|g\|\,\|\widetilde X-X_\lambda\|.
		\end{equation}
		Consequently,
		\begin{equation}\label{eq:sgr-distance-value-bound}
			\bigl|
			\widehat\varphi(\lambda;W)-\varphi(\lambda)
			\bigr|
			\leq
			\left(
			\|g\|\,
			\operatorname{dist}(\widetilde X,\Omega(\lambda))
			\right)^{1/2}.
		\end{equation}
		
		\item Suppose that, locally and uniformly for
		$\lambda\in\mathcal N$, there exist $\kappa>0$ and
		$\theta\in[0,1]$ such that
		\begin{equation}\label{eq:inner-holder-error-bound}
			\operatorname{dist}(\widetilde X,\Omega(\lambda))
			\leq
			\kappa\|g\|^\theta.
		\end{equation}
		Set
		\(
		\beta:=\frac{1+\theta}{2}.
		\)
		Then
		\begin{equation}\label{eq:sgr-value-holder}
			\bigl|
			\widehat\varphi(\lambda;W)-\varphi(\lambda)
			\bigr|
			\leq
			\sqrt{\kappa}\,\|g\|^\beta.
		\end{equation}
		If $\widetilde X$ remains in a bounded set, then there exists
		a constant $C>0$ such that
		\begin{equation}\label{eq:rsgr-value-holder}
			\bigl|
			\widehat\varphi(\lambda;W)-\varphi(\lambda)
			\bigr|
			\leq
			C\|\operatorname{rSGR}_\lambda(W)\|^\beta.
		\end{equation}
		Moreover, if the constraint nondegeneracy condition
		\eqref{eq:cn_nnlsdc} holds at $\lambda^*$, then, after shrinking
		$\mathcal N$ if necessary, the error bound
		\eqref{eq:inner-holder-error-bound} holds locally with
		$\theta=1$, and hence $\beta=1$.
		
		\item At the $k$th outer iteration, let
		\[
		\widetilde X^k
		:=
		\operatorname{PG}_{\lambda^k}(W^k),
		\qquad
		\widehat\varphi_k
		:=
		\|\mathcal A\widetilde X^k-b\|,
		\]
		and define
		\[
		\eta_k
		:=
		\|\operatorname{rSGR}_{\lambda^k}(W^k)\|,
		\qquad
		e_k:=\lambda^k-\lambda^*,
		\qquad
		\Theta_k
		:=
		|e_k|^\gamma+|e_{k-1}|^\gamma.
		\]
		Suppose that the H\"older error bound
		\eqref{eq:inner-holder-error-bound} holds uniformly along these
		inner iterates with exponent $\theta$, and that
		$\{\widetilde X^k\}$ is bounded. If
		\begin{equation}\label{eq:sgr-inexact-forcing}
			\frac{\eta_k^\beta}{|e_k|}
			+
			\frac{\eta_k^\beta+\eta_{k-1}^\beta}
			{|\lambda^k-\lambda^{k-1}|}
			=
			O(\Theta_k),
			\qquad
			\beta=\frac{1+\theta}{2},
		\end{equation}
		whenever the denominators are nonzero, then the value errors
		\[
		\xi_k
		:=
		\widehat\varphi_k-\varphi(\lambda^k)
		\]
		satisfy the inexact secant condition
		\eqref{eq:inexact-condition}. Therefore, all the local convergence
		conclusions of Proposition~\ref{prop:inexact-sec-converge} hold.
	\end{enumerate}
\end{proposition}

\begin{proof}
	By Proposition~\ref{prop:sgr}(ii),
	\[
	g\in\partial f_\lambda(\widetilde X).
	\]
	Thus, there exists
	$\widetilde Z\in\partial\|\widetilde X\|_*$ such that
	\[
	g
	=
	\lambda\widetilde Z
	+
	\mathcal A^*(\mathcal A\widetilde X-b).
	\]
	For any $X_\lambda\in\Omega(\lambda)$, choose
	$Z_\lambda\in\partial\|X_\lambda\|_*$ satisfying
	\[
	0
	=
	\lambda Z_\lambda
	+
	\mathcal A^*(\mathcal A X_\lambda-b).
	\]
	Using the monotonicity of the subdifferential of the nuclear norm, we
	obtain
	\[
	\begin{aligned}
		\langle g,\widetilde X-X_\lambda\rangle
		&=
		\lambda
		\langle
		\widetilde Z-Z_\lambda,
		\widetilde X-X_\lambda
		\rangle
		+
		\|\mathcal A(\widetilde X-X_\lambda)\|^2\\
		&\geq
		\|\mathcal A(\widetilde X-X_\lambda)\|^2.
	\end{aligned}
	\]
	On the other hand,
	\[
	\begin{aligned}
		\bigl|
		\widehat\varphi(\lambda;W)-\varphi(\lambda)
		\bigr|
		&=
		\left|
		\|\mathcal A\widetilde X-b\|
		-
		\|\mathcal A X_\lambda-b\|
		\right|\\
		&\leq
		\|\mathcal A(\widetilde X-X_\lambda)\|.
	\end{aligned}
	\]
	Consequently,
	\[
	\bigl|
	\widehat\varphi(\lambda;W)-\varphi(\lambda)
	\bigr|^2
	\leq
	\|\mathcal A(\widetilde X-X_\lambda)\|^2
	\leq
	\langle g,\widetilde X-X_\lambda\rangle
	\leq
	\|g\|\,\|\widetilde X-X_\lambda\|,
	\]
	which proves \eqref{eq:sgr-basic-value-bound}. Taking the infimum over
	$X_\lambda\in\Omega(\lambda)$ gives
	\eqref{eq:sgr-distance-value-bound}.
	
	Combining \eqref{eq:sgr-distance-value-bound} with the H\"older error
	bound \eqref{eq:inner-holder-error-bound} yields
	\[
		\bigl|
		\widehat\varphi(\lambda;W)-\varphi(\lambda)
		\bigr|
		\leq
		\left(
		\kappa\|g\|^{1+\theta}
		\right)^{1/2}=
		\sqrt{\kappa}\,\|g\|^\beta,
	\]
	which proves \eqref{eq:sgr-value-holder}. Moreover,
	\[
	\|g\|
	=
	L_{\mathcal A^*\mathcal A}
	\bigl(1+\|\widetilde X\|\bigr)
	\|\operatorname{rSGR}_\lambda(W)\|.
	\]
	Hence, if $\widetilde X$ remains bounded, then
	\eqref{eq:rsgr-value-holder} follows with a uniform constant $C>0$.
	
	It remains to justify the linear error bound under constraint
	nondegeneracy. Define
	\[
	\widetilde y:=b-\mathcal A\widetilde X,
	\qquad
	\widetilde s:=\|\widetilde X\|_*.
	\]
	The subgradient relation for $g$ can be equivalently written as
	\begin{equation}\label{eq:perturbed-dual-relation}
		\mathcal A^*\widetilde y
		=
		\lambda\widetilde Z-g.
	\end{equation}
	Since
	$\widetilde Z\in\partial\|\widetilde X\|_*$, we have
	\[
	\|\widetilde Z\|_2\leq1,
	\qquad
	\langle\widetilde Z,\widetilde X\rangle
	=
	\|\widetilde X\|_*
	=
	\widetilde s.
	\]
	Therefore,
	\[
	(\lambda,\lambda\widetilde Z)
	\in
	\operatorname{epi}\|\cdot\|_2,
	\qquad
	(\widetilde s,-\widetilde X)
	\in
	\operatorname{epi}\|\cdot\|_*
	=
	\bigl(\operatorname{epi}\|\cdot\|_2\bigr)^*,
	\]
	and
	\[
	\left\langle
	(\lambda,\lambda\widetilde Z),
	(\widetilde s,-\widetilde X)
	\right\rangle
	=
	\lambda\widetilde s
	-
	\lambda
	\langle\widetilde Z,\widetilde X\rangle
	=
	0.
	\]
	The projection characterization for a closed convex cone therefore
	gives
	\begin{equation}\label{eq:exact-projection-from-subgradient}
		\Pi_{\operatorname{epi}\|\cdot\|_2}
		\left(
		(\lambda,\lambda\widetilde Z)
		-
		(\widetilde s,-\widetilde X)
		\right)
		=
		(\lambda,\lambda\widetilde Z).
	\end{equation}
	
	Notice also that
	\[
	\widetilde y-b+\mathcal A\widetilde X=0.
	\]
	Hence, using \eqref{eq:exact-projection-from-subgradient} and the
	nonexpansiveness of the metric projection, we obtain
	\begin{align}
		&
		\|G(\lambda,\widetilde y,\widetilde s,\widetilde X)\|
		\nonumber\\
		={}&
		\left\|
		(\lambda,\mathcal A^*\widetilde y)
		-
		\Pi_{\operatorname{epi}\|\cdot\|_2}
		\left(
		(\lambda,\mathcal A^*\widetilde y)
		-
		(\widetilde s,-\widetilde X)
		\right)
		\right\|
		\nonumber\\
		\leqslant{}&
		\left\|
		(\lambda,\mathcal A^*\widetilde y)
		-
		(\lambda,\lambda\widetilde Z)
		\right\|
		\nonumber\\
		&\quad+
		\left\|
		\Pi_{\operatorname{epi}\|\cdot\|_2}
		\left(
		(\lambda,\lambda\widetilde Z)
		-
		(\widetilde s,-\widetilde X)
		\right)
		-
		\Pi_{\operatorname{epi}\|\cdot\|_2}
		\left(
		(\lambda,\mathcal A^*\widetilde y)
		-
		(\widetilde s,-\widetilde X)
		\right)
		\right\|
		\nonumber\\
		\leqslant{}&
		2\|\mathcal A^*\widetilde y-\lambda\widetilde Z\|
		=
		2\|g\|.
		\label{eq:G-controlled-by-sgr}
	\end{align}
	The nonsingularity property used in the proof of
	Proposition~\ref{prop:sss} implies a local Lipschitz error bound for
	the KKT equation. Therefore, after shrinking $\mathcal N$ if
	necessary, there exists a constant $c>0$ such that
	\[
	\operatorname{dist}(\widetilde X,\Omega(\lambda))
	\leq
	c\|G(\lambda,\widetilde y,\widetilde s,\widetilde X)\|
	\leq
	2c\|g\|.
	\]
	Thus \eqref{eq:inner-holder-error-bound} holds locally with
	$\theta=1$.
	
	Finally, the assumptions in part (iii), together with
	\eqref{eq:rsgr-value-holder}, imply that
	\[
	|\xi_k|
	\leqslant
	C\eta_k^\beta
	\quad  \text{and}\quad
	|\xi_k-\xi_{k-1}|
	\leqslant
	C\bigl(\eta_k^\beta+\eta_{k-1}^\beta\bigr).
	\]
	Condition \eqref{eq:sgr-inexact-forcing} therefore gives
	\[
	\frac{|\xi_k|}{|e_k|}
	+
	\frac{|\xi_k-\xi_{k-1}|}
	{|\lambda^k-\lambda^{k-1}|}
	=
	O\bigl(|e_k|^\gamma+|e_{k-1}|^\gamma\bigr),
	\]
	which is precisely \eqref{eq:inexact-condition}. The conclusions then
	follow from Proposition~\ref{prop:inexact-sec-converge}. This
	completes the proof.
\end{proof}

Combining the proximal generation method Algorithm \ref{alg:RS} and the secant method \eqref{eq:sec1}, we obtain our proximal generation based level set method with secant iterations in Algorithm \ref{alg:rsls-sec}.  Practical implementation details are discussed in the next section.

For notational simplicity, in Algorithm~\ref{alg:rsls-sec} we suppress
the distinction between the exact and computed value functions.
Specifically, whenever \(X^k\) is returned by the proximal generation
method, the notation \(\varphi(\lambda^k)\) in the algorithm refers to
the computed residual
\(
\varphi(\lambda^k):=\|\mathcal A X^k-b\|.
\)
Thus, this quantity corresponds to \(\widehat\varphi_k\) in the preceding
inexact analysis. The inexactness conditions derived above are sufficient
asymptotic conditions for preserving the local convergence rate and are
not imposed explicitly in Algorithm~\ref{alg:rsls-sec}.

\begin{remark}
	In Algorithm \ref{alg:rsls-sec}, we initialize $\lambda^{-1}$ and $\lambda^{0}$ near the upper bound $\|\cA^{*}b\|_{2}$. Although smaller starting values still converge, they often induce higher rank subproblems and extra cost. Starting near $\|\cA^{*}b\|_{2}$ yields lower rank iterations, allowing us to grow the factor matrices $U$ and $V$ incrementally and thus reduce computation.
\end{remark}
\begin{algorithm}[!h]
	\caption{A proximal generation based level set method with secant iterations for \eqref{eq:lscnnm}}
	\label{alg:rsls-sec}
	\begin{algorithmic}[1]
		\Require Initial points and bounds for $\lambda$: $0 \leqslant \lambda_{\text{lo}} \leqslant \lambda^{-1} < \lambda^0 \leqslant \lambda_{\text{up}} \leqslant \|\cA^*b\|_2$, corresponding initial residual values $\varphi(\lambda^0)>\varrho>\varphi(\lambda^{-1})$, tolerances $\varepsilon_{\text{sec}}>0$, initial point $X^0=U_{X^0}\Sigma_{X^0}V^T_{X^0}$, initial value $\varphi(\lambda^0)$, and $k=0$.
		\While{$|\varphi(\lambda^k)-\varrho|>\varepsilon_{\text{sec}}$}
		\State Generate $\lambda^{k+1}$ with some globally convergent secant method, e.g. \cite[Alg 5.2]{Li2024} .
		\State Compute $X^{k+1}=U_{X^{k+1}}\Sigma_{X^{k+1}}V^T_{X^{k+1}}$ by the proximal generation method (Algorithm \ref{alg:RS}) with initial point $X^{k}=U_{X^{k}}\Sigma_{X^{k}}V^T_{X^{k}}$; set $\varphi(\lambda^{k+1}):=
		\|\mathcal A X^{k+1}-b\|$.
		\State $k\leftarrow k+1$.
		\EndWhile
		
		\State Return $\lambda^* = \lambda^k$ and $X(\lambda^*)$.
	\end{algorithmic}
\end{algorithm}
\section{Numerical experiments}
In this section, we showcase the numerical performance of our method. We tested \eqref{eq:lscnnm} from low-rank regressions and low-rank matrix completions. In our experiments, we measure the accuracy of the obtained $X$ and $\lambda$ by
\begin{equation}\label{eq:LSeta}
	\eta:=\max \left\lbrace \frac{|\left\|\cA X-b \right\|-\varrho |}{\max \left\lbrace 1,\varrho \right\rbrace },	 \|\operatorname{rSGR}_\lambda(X)\| \right\rbrace . 
\end{equation}
We compare our proximal generation based level set method using secant iterations ($\text{GLS}_\text{sec}$) with the proximal generation based level set method using only bisection iterations ($\text{GLS}_\text{bis}$), as well as with  SPGL1\footnote{\url{https://friedlander.io/publications/2010-sparse-optimization-with-least-squares/}.} \cite{VandenBerg2011} and ADMM,
 all with a tolerance of $10^{-3}$. For SPGL1, which has complicated stopping conditions, we set \verb|options.optTol=1e-3| and use the \verb|SPOR-SPG-H| mode, which is the faster option according to \cite[Section 8]{VandenBerg2011}. 
All numerical experiments were implemented using MATLAB R2023a (version 9.14) on Ubuntu 22.04.3 LTS, running on an Intel Core i9-10900 CPU operating at 2.80GHz and supported by 64GB of RAM.
\subsection{Numerical Issues in Algorithm \ref{alg:rsls-sec}}\label{sec:numeric-issue}

To achieve fast local convergence of the secant method, we use initial bisection steps to obtain a good starting point. The algorithm switches to the secant iteration \eqref{eq:sec1} when $
\frac{|\varphi(\lambda^k) - \varrho|}{\max\{1,\varrho\}} \leqslant 0.1$,
and reverts to bisection if stagnation occurs.

The parameter \(k_{\text{check}}\) in Algorithm \ref{alg:RS} is set adaptively: initially large (e.g., 40) to reduce computational cost, then small (e.g., 2--4) once the relative error falls below \(10^{-2}\). Iterations stop if the residual stagnates or a limit is reached.

For efficiency, proximal gradient steps use a partial SVD (PROPACK), which requests slightly more singular values than the previous rank. The alternating minimization subproblems solve two positive definite systems: via Cholesky decomposition for small-scale matrices or the preconditioned quasi-minimal residual method otherwise. For matrix completion with \(\operatorname{Id}-\cA^*\cA \succeq 0\), we solve a majorized version of \(g_\lambda(L,R)\).

\subsection{ADMM algorithms for \eqref{eq:lscnnm}}

For matrix completion problems where the operator $\cA\cA^*$ is an identity mapping, we apply ADMM to the dual problem of \eqref{eq:lscnnm}:
\begin{equation}\label{eq:dual1lscnnm}
		\min_{s\in\bR^p,T\in\bR^{m\times n}}\delta_{\|\cdot\|_2\leqslant 1}(T)+\varrho\|s\|+\left\langle b,s\right\rangle 
		\text{ s.t. } \cA^*s+T=0.
\end{equation}
In the numerical test, we apply the relaxed ADMM in \cite[Algorithm 1.1]{Sun2024} with the relaxed parameter $\rho=1.8$, and adjust the penalty parameter $\sigma$ by the primal-dual feasibility criterion similar to the approach in \cite{Yang2021}. This configuration performs better in our preliminary test. 

For low-rank regression problems, we apply ADMM to the equivalent form of \eqref{eq:dual1lscnnm}
\begin{equation}\label{eq:dual2lscnnm}
\begin{aligned}
\min_{r,s\in\bR^p,T\in\bR^{m\times n}} \quad
& \delta_{\|\cdot\|_2\leqslant 1}(T)+\varrho\|r\|+\left\langle b,s\right\rangle \\
\text{s.t.}\quad
& \cA^*s+T=0,\quad r-s=0,
\end{aligned}
\end{equation}
and set $(r,T)$ as one block and $s$ as the other. In the numerical test, we apply ADMM \cite[Appendix B]{Fazel2013} with dual step length $\tau = 1.618$, and adjust the penalty parameter $\sigma$ based on the primal-dual feasibility criterion. This tailored configuration has yielded improved performance based on our preliminary testing.

To keep the stopping criteria consistent with \eqref{eq:LSeta}, we assess the optimality of ADMM using
\begin{equation}\label{eq:admmres_lscnnm}
	\frac{|\|\cA X-b\|-\varrho |}{\max \{ 1,\varrho \}} \quad \text{and} \quad \|\operatorname{rSGR}_\lambda(X)\|,
\end{equation}
where $\operatorname{PG}_\lambda(X)$ in \eqref{eq:sgr} and \eqref{eq:rsgr} is replaced by $\operatorname{PG}_{\varrho/\|s\|}(X)$ and $\operatorname{PG}_{\varrho/\|r\|}(X)$ for matrix completion and low-rank regression, respectively, as implied by their KKT systems.

\begin{remark}
	Recent advances have introduced successful accelerated ADMM-type methods \cite{Sun2024,Chen2024}. However, in our assessments, these accelerated versions did not exhibit superiority. This observation may be attributed to the nonpolyhedral components present in our problem formulations.
\end{remark}
\subsection{Matrix regression problems}
Low-rank matrix regression is an important application of \eqref{eq:lscnnm}. We collected several datasets, mainly from bioinformatics. In these examples, the linear map $\cA$ is given explicitly as a matrix, and $\cA X$ is implemented as matrix multiplication after vectorization. As in \cite{Huang2010} and \cite{Li2018a}, we expand the matrix $\cA$ using polynomial basis functions over its columns. Instances denoted by names ending with 'e$k$' are expanded using a $k$-order polynomial. 
Moreover, we also expand the rows in some instances. We summarize the statistics of our test data in Table \ref{table:lrmr_stat}. 
The breast cancer dataset referenced in \cite{Chin2006} can be accessed via the PMA package \cite{Witten2009}. This dataset is organized into distinct instances based on chromosomal divisions.
The yeast dataset from \cite{Spellman1998} is available via the spls package \cite{Chun2010}.
 Another dataset is from \cite{Qin2021}. The parameter $\varrho$ in \eqref{eq:lscnnm} is set by $\varrho=c\|b\|$ where the coefficient $c$ is listed in Table \ref{table:lrmr_stat}.
\begin{table}[t]
	\centering
	\caption{Statistics of test data in low-rank regressions}
	\begin{tabular}{cccccccc}
		\hline
		Prob idx & Name               & Source       & $m$     & $n$     & $p$        &$\|b\|$  & $c=\frac{\varrho}{\|b\|}$ \\ \hline
		1             & B\_C1to5           &   \cite{Chin2006}          & 601   & 5930  & 527770   & 5.50+3 & 0.08  \\
		2             & B\_C4              &   \cite{Chin2006}          & 167   & 715   & 63635    & 1.71+3 & 0.08  \\
		3             & B\_C6to10          &   \cite{Chin2006}          & 609   & 4084  & 363476   & 3.86+3 & 0.08  \\
		4             & B\_C7              &   \cite{Chin2006}          & 161   & 923   & 82147    & 2.00+3 & 0.08  \\
		5             & B\_C11             &  \cite{Chin2006}           & 179   & 1122  & 99858    & 2.25+3 & 0.08  \\
		6             & B\_C11to20         &   \cite{Chin2006}          & 822   & 7778  & 692242   & 5.58+3 & 0.08  \\
		7             & B\_C20plus         &    \cite{Chin2006}         & 173   & 1247  & 110983   & 2.29+3 & 0.08  \\
		8             & B\_C7e2            &    \cite{Chin2006}         & 13203 & 2767  & 246263   & 3.44+3 & 0.08  \\
		9             & B\_C20pluse2       &    \cite{Chin2006}         & 15225 & 3739  & 332771   & 3.91+3 & 0.08  \\
		10            & bt                 &   \cite{Qin2021}          & 939   & 4000  & 980000   & 3.61+0   & 0.15  \\
		11            & srf\_o3            &     \cite{Qin2021}         & 2042  & 3199  & 10809421 & 1.02+0     & 0.6   \\
		12            & yeaste2             &    \cite{Spellman1998}         & 5778  & 919   & 166339   & 1.86+2 & 0.08  \\
		13            & yeaste3             &    \cite{Spellman1998}         & 44362 & 1837  & 332497   & 2.63+2 & 0.15  \\
		14            & B\_CWhole          &      \cite{Chin2006}       & 2149  & 19672 & 1750808  & 9.13+3 & 0.08  \\
		15            & test1\_e3          &      \cite{Qin2021}        & 4556  & 2991  & 149550   & 2.63+1  & 0.08  \\
		16            & test1\_e4          &    \cite{Qin2021}         & 31961 & 2991  & 149550   & 2.63+1  & 0.08  \\ 
		17            & B\_C1to5e2         &    \cite{Chin2006}         & 5711  & 5930  & 527770   & 5.50+3 & 0.08  \\
		18            & B\_C6to10e2        &  \cite{Chin2006}           & 19733 & 4084  & 363476   & 3.86+3 & 0.08  \\
		19            & B\_C11to20e2     &    \cite{Chin2006}         & 25748 & 7778  & 692242   & 5.58+3 & 0.08  \\
		20            & bte2               &    \cite{Qin2021}          & 25058 & 4000  & 980000   & 3.61+0   & 0.35  \\
		21            & B\_C11e2           &    \cite{Chin2006}         & 16290 & 5606  & 498934   & 4.98+3 & 0.08  \\ \hline
	\end{tabular}
	\label{table:lrmr_stat}
\end{table}
\begin{table}[!h]
		\centering
	\caption{The performance of proximal generation based level set method with secant iterations ($\text{GLS}_\text{sec}$), SPGL1(SPG), and ADMM in solving the least-squares constrained nuclear norm minimization \eqref{eq:lscnnm} with the $\varrho=c\|b\|$ in Table \ref{table:lrmr_stat}}
\begin{tabular}{ccccc}
	\hline
	\multirow{2}{*}{idx} & \multirow{2}{*}{rank} & time(s)                                                   &  & residual                              \\
	&                       & $\text{GLS}_\text{sec}$ : SPG : ADMM                                             &  & $\text{GLS}_\text{sec}$ : SPG : ADMM                         \\ \hline
	1                    & 11                    & 4.36+0 $\mid$ 3.01+2 $\mid$ 1.26+1                         &  & 8.89-4 $\mid$ 9.89-4 $\mid$ 9.68-4 \\
	2                    & 25                    & 7.02-1 $\mid$ 7.70+0 $\mid$ 8.89-1                         &  & 1.23-4 $\mid$ 1.31-4 $\mid$ 1.79-4 \\
	3                    & 13                    & 3.85+0 $\mid$ 2.49+2 $\mid$ 5.88+0                         &  & 3.97-5 $\mid$ 8.63-4 $\mid$ 7.85-4 \\
	4                    & 17                    & 4.80-1 $\mid$ 7.88+0 $\mid$ 7.78-1                         &  & 8.59-4 $\mid$ 1.04-4 $\mid$ 9.83-4 \\
	5                    & 16                    & 6.40-1 $\mid$ 1.42+1 $\mid$ 7.88-1                         &  & 2.20-4 $\mid$ 8.43-4 $\mid$ 9.06-4 \\
	6                    & 11                    & 4.44+0 $\mid$ 3.24+2 $\mid$ 3.87+1                         &  & 5.93-4 $\mid$ 8.63-4 $\mid$ 9.93-4 \\
	7                    & 14                    & 2.63-1 $\mid$ 1.52+1 $\mid$ 7.72-1                         &  & 7.19-4 $\mid$ 2.15-4 $\mid$ 7.49-4 \\
	8                    & 3                     & 3.76+1 $\mid$ 2.74+3 $\mid$ 2.04+2                         &  & 3.27-4 $\mid$ 8.88-4 $\mid$ 2.56-4 \\
	9                    & 5                     & 2.60+1 $\mid$ \textgreater{}10800 $\mid$ 5.73+2             &  & 7.99-4 $\mid$ 6.59-1 $\mid$ 1.20-4 \\
	10                   & 200                   & 1.09+2 $\mid$ 3.82+2 $\mid$ 9.72+2                         &  & 7.28-4 $\mid$ 4.19-4 $\mid$ 5.20-4 \\
	11                   & 10                    & 1.08+1 $\mid$ 2.70+2 $\mid$ 2.53+3                         &  & 5.73-4 $\mid$ 7.56-4 $\mid$ 9.99-4 \\
	12                   & 17                    & 1.55+1 $\mid$ 7.16+2 $\mid$ 1.18+1                         &  & 1.95-4 $\mid$ 4.45-4 $\mid$ 4.59-4 \\
	13                   & 17                    & 3.71+2 $\mid$ \textgreater{}10800 $\mid$ 4.25+2             &  & 2.27-4 $\mid$ 5.69-1 $\mid$ 7.26-4 \\
	14                   & 10                    & 2.54+1 $\mid$ 6.80+3 $\mid$ 5.04+2                         &  & 9.86-4 $\mid$ 9.29-4 $\mid$ 9.99-4 \\
	15                   & 50                    & 9.54+1 $\mid$ \textgreater{}10800 $\mid$ 7.36+2             &  & 2.79-4 $\mid$ 1.24-2 $\mid$ 2.09-4 \\
	16                   & 50                    & 6.76+2 $\mid$ \textgreater{}10800 $\mid$ 5.38+3             &  & 2.65-4 $\mid$ 8.40-1 $\mid$ 1.64-4 \\
	17                   & 9                     & 4.42+1 $\mid$ 3.95+3 $\mid$ 2.72+2                         &  & 6.75-4 $\mid$ 7.50-4 $\mid$ 3.43-4 \\
	18                   & 13                    & 1.20+2 $\mid$ \textgreater{}10800 $\mid$ 9.80+2             &  & 8.92-4 $\mid$ 1.84-1 $\mid$ 1.79-4 \\
	19                   & 12                    & 2.01+2 $\mid$ \textgreater{}10800 $\mid$ 2.95+3             &  & 4.52-4 $\mid$ 7.85-1 $\mid$ 5.05-4 \\
	20                   & 103                   & 1.13+3 $\mid$ \textgreater{}10800 $\mid$ \textgreater{}10800 &  & 5.24-4 $\mid$ 5.02-1 $\mid$ 2.17-3 \\
	21                   & 5                     & 1.10+2 $\mid$ \textgreater{}10800 $\mid$ 5.57+2             &  & 4.88-5 $\mid$ 5.11-1 $\mid$ 9.82-4 \\ \hline
\end{tabular}
\label{table:lrmr_test}
\end{table}

We conduct a comparative analysis of our  $\text{GLS}_\text{sec}$ against SPGL1 and ADMM for solving least-squares constrained problems. 
 The initial bounds for $\lambda$ are set as $\lambda_{\text{lo}}=0$ and $\lambda_{\text{up}}=\|\cA^*b\|_2$, with an initial value of $\lambda^1=10^{-2}\|\cA^*b\|_2$. Additionally, we ran 20 iterations of the APG method at $\lambda^1$ to generate an initial point, which was then used for all evaluated algorithms. The time limit for each algorithm was 10,800 seconds.

The test results are presented in Table \ref{table:lrmr_test}. The second column represents the rank of $X$ computed by our method. The residual, denoted as $\eta$ in \eqref{eq:LSeta} for the level set method and ADMM when solving \eqref{eq:lscnnm}, is also presented. For SPGL1, the residual is characterized by \verb|Rel Error| in its log, which is actually the first term in \eqref{eq:LSeta}. In cases where SPGL1 or ADMM cannot solve \eqref{eq:lscnnm} within the time limit, we report the residual of their solution from the final iteration. 

The results presented in Table \ref{table:lrmr_test} demonstrate the superiority of our $\text{GLS}_\text{sec}$. Compared
 with other methods, both $\text{GLS}_\text{sec}$ successfully solve all instances, whereas SPGL1 fails to solve many instances within the time limit. For these unsolved cases, the solutions returned by SPGL1 at the maximum time are significantly below the desired accuracy. For instances that SPGL1 can solve, $\text{GLS}_\text{sec}$ is substantially faster, with speedups reaching nearly 300 times. Although ADMM also fails to solve all instances, it is more efficient than SPGL1 and even slightly outperforms $\text{GLS}_\text{sec}$ in Problem 12. For most problems, $\text{GLS}_\text{sec}$ is more efficient than ADMM with speedups exceeding 200 times in some cases.
\begin{table}[!h]
		\centering
	\caption{The performance of proximal generation based level set method with secant iterations ($\text{GLS}_\text{sec}$) and with bisection iterations ($\text{GLS}_\text{bis}$)}
\begin{tabular}{cccccccccc}
	\hline
	\multirow{2}{*}{idx} & \multicolumn{4}{c}{$\text{GLS}_\text{sec}$}               &  & \multicolumn{4}{c}{$\text{GLS}_\text{bis}$}               \\
	& $\lambda$      & iter & time(s)     & residual &  & $\lambda$      & iter & time(s)     & residual \\ \hline
	1                    & 3.216+2 & 3:2  & 4.363+0 & 8.893-4 &  & 3.205+2 & 9    & 8.323+0 & 6.605-4 \\
	2                    & 2.062+1 & 3:3  & 7.016-1 & 1.226-4 &  & 2.069+1 & 9    & 1.528+0 & 8.242-4 \\
	3                    & 2.379+2 & 3:3  & 3.846+0 & 3.972-5 &  & 2.381+2 & 9    & 5.929+0 & 1.571-4 \\
	4                    & 3.158+1 & 3:2  & 4.803-1 & 8.587-4 &  & 3.153+1 & 10   & 1.031+0 & 2.944-4 \\
	5                    & 5.164+1 & 3:3  & 6.403-1 & 2.198-4 &  & 5.170+1 & 11   & 1.065+0 & 5.106-5 \\
	6                    & 4.729+2 & 2:1  & 4.435+0 & 5.925-4 &  & 4.886+2 & 9    & 1.678+1 & 7.470-4 \\
	7                    & 5.980+1 & 3:1  & 2.634-1 & 7.193-4 &  & 6.067+1 & 7    & 6.215-1 & 9.378-4 \\
	8                    & 7.569+2 & 7:2  & 3.758+1 & 3.274-4 &  & 7.563+2 & 12   & 5.177+1 & 4.301-4 \\
	9                    & 1.087+3 & 2:1  & 2.601+1 & 7.993-4 &  & 1.087+3 & 9    & 6.424+1 & 4.096-4 \\
	10                   & 1.472-3 & 4:3  & 1.087+2 & 7.276-4 &  & 1.467-3 & 8    & 1.211+2 & 4.527-4 \\
	11                   & 1.256-3 & 2:1  & 1.081+1 & 5.726-4 &  & 1.669-3 & 4    & 2.051+1 & 6.659-4 \\
	12                   & 2.895+1 & 3:1  & 1.547+1 & 1.950-4 &  & 2.896+1 & 7    & 2.145+1 & 5.221-4 \\
	13                   & 1.220+2 & 7:1  & 3.713+2 & 2.268-4 &  & 1.219+2 & 12   & 4.493+2 & 3.469-4 \\
	14                   & 1.283+3 & 2:1  & 2.543+1 & 9.858-4 &  & 1.383+3 & 9    & 1.230+2 & 5.632-4 \\
	15                   & 6.092-2 & 7:2  & 9.543+1 & 2.792-4 &  & 6.097-2 & 15   & 1.588+2 & 6.340-4 \\
	16                   & 6.097-2 & 7:2  & 6.760+2 & 2.653-4 &  & 6.097-2 & 15   & 1.081+3 & 2.362-4 \\
	17                   & 5.048+2 & 2:3  & 4.415+1 & 6.746-4 &  & 4.324+2 & 3    & 1.907+1 & 2.333-5 \\
	18                   & 5.409+2 & 2:3  & 1.196+2 & 8.917-4 &  & 5.382+2 & 8    & 2.151+2 & 5.412-5 \\
	19                   & 9.876+2 & 2:2  & 2.011+2 & 4.521-4 &  & 9.938+2 & 7    & 4.395+2 & 9.809-4 \\
	20                   & 7.133-3 & 11:4 & 1.126+3 & 5.242-4 &  & 7.154-3 & 17   & 1.362+3 & 2.840-4 \\
	21                   & 1.229+3 & 3:2  & 1.103+2 & 4.883-5 &  & 1.225+3 & 8    & 1.380+2 & 7.543-4 \\ \hline
\end{tabular}
\label{table:lrls-sec-bis}
\end{table}

The comparison between the secant method and the bisection method is detailed in Table \ref{table:lrls-sec-bis}. The second and sixth columns display the solution $\lambda$ of the univariate nonlinear equation \eqref{eq:une} as obtained by the secant and bisection methods, respectively. The third column lists the number of iterations for both the bisection and the secant methods in $\text{GLS}_\text{sec}$, while the seventh column shows the iterations for the pure bisection method. Compared with the bisection variant $\text{GLS}_\text{bis}$, $\text{GLS}_\text{sec}$ requires fewer iterations and achieves a speedup of approximately 2 to 5 times, highlighting the effectiveness of our secant approach.
\subsection{Matrix completion problems}
We also test our method on the classic matrix completion problems. We use MovieLens datasets \cite{Harper2015} and Jester joke datasets \cite{Goldberg2001}.
We augmented the original Jester datasets by generating additional synthetic ratings, obtained by averaging randomly selected users’ ratings and adding small perturbations.
We also generate synthetic data following the procedure of \cite{Toh2010}. We summarize the statistics of these data in Table \ref{table:lrmc_stat}. The $\varrho$ in \eqref{eq:lscnnm} is set by $\varrho=c\|b\|$ where the coefficient $c$ is in the last column of \ref{table:lrmc_stat}.
\begin{table}[t]
	\centering
	\caption{Statistics of test data in low-rank matrix completion}
	\begin{tabular}{cccccccc}
		\hline
		Prob   idx & Name      & Source & $m$      & $n$     & $p$        & $\|b\|$  & $c=\frac{\varrho}{\|b\|}$ \\ \hline
		1             & u100K     &   \cite{Harper2015}    & 943    & 1682  & 49918    & 8.27+2 & 0.2   \\
		2             & uleastsmall       &   \cite{Harper2015}    & 610    & 9724  & 53246    & 8.39+2 & 0.25  \\
		3             & u1M       &    \cite{Harper2015}   & 6040   & 3706  & 498742   & 2.65+3 & 0.25  \\
		4             & u20m      &  \cite{Harper2015}     & 7120   & 14026 & 525002   & 2.67+3 & 0.25  \\
		5             & u10M      &  \cite{Harper2015}     & 71567  & 10677 & 4983232  & 8.19+3 & 0.2   \\
		6            & u25m      & \cite{Harper2015}      & 162541 & 59047 & 12472784 & 1.30+4 & 0.25  \\
		7            & ulatest   &   \cite{Harper2015}    & 330975 & 83239 & 17198680 & 1.54+4 & 0.25  \\
		8             & jester1ex &  \cite{Goldberg2001}     & 24983  & 10100 & 249830   & 3.35+3 & 0.3   \\
		9             & jester2ex &   \cite{Goldberg2001}    & 23500  & 10100 & 235000   & 3.23+3 & 0.3   \\
		10             & jester3ex &   \cite{Goldberg2001}    & 24938  & 9926  & 265164   & 4.27+3 & 0.3   \\
		11            & jester4ex &  \cite{Goldberg2001}     & 54905  & 6653  & 553353   & 5.91+3 & 0.3   \\
		12            & jester5ex &   \cite{Goldberg2001}    & 7699   & 9011  & 88920    & 2.33+3 & 0.2   \\
		13            & Rand1     &   \cite{Toh2010}    & 7000   & 8000  & 2992002  & 1.22+4 & 0.2   \\
		14            & Rand2     &    \cite{Toh2010}   & 7000   & 10000 & 3385373  & 1.30+4 & 0.2   \\
		15            & Rand3     &   \cite{Toh2010}    & 7500   & 12000 & 4280292  & 1.54+4 & 0.2   \\
		16            & Rand4     &  \cite{Toh2010}     & 7500   & 10000 & 4182480  & 1.59+4 & 0.2   \\
		17            & Rand5     &   \cite{Toh2010}    & 9000   & 14000 & 3444147  & 1.31+4 & 0.2   \\
		18            & Rand6     &   \cite{Toh2010}    & 10000  & 12000 & 3947772  & 1.54+4 & 0.25  \\
		19            & Rand7     &  \cite{Toh2010}     & 10000  & 12500 & 4037941  & 1.56+4 & 0.25  \\
		20            & Rand8     &  \cite{Toh2010}     & 20000  & 15000 & 5245442  & 1.62+4 & 0.25  \\
		21            & Rand9     & \cite{Toh2010}      & 30000  & 31000 & 10972303 & 2.57+4 & 0.25  \\
		22            & Rand10    &  \cite{Toh2010}     & 50000  & 8000  & 10433300 & 2.50+4 & 0.25  \\ \hline
	\end{tabular}
	\label{table:lrmc_stat}
\end{table}

Except where noted, settings for matrix completion follow those for matrix regression. For our $\text{GLS}_\text{sec}$ and $\text{GLS}_\text{bis}$, we set $\lambda^1=0.1\times\|\cA^*b\|_2$ for (semi)real instances and $\lambda^1=0.3\times\|\cA^*b\|_2$ for synthetic instances. In alternating minimization, we adopt the aforementioned majorization. Accordingly, we use a larger $k_\text{check}$ for line 6 in Algorithm \ref{alg:RS}.

The results of the matrix completion problems are presented in Table \ref{table:lrmc_test}. Our method $\text{GLS}_\text{sec}$ continues to demonstrate its superiority in this context.  When evaluating SPGL1 and ADMM for solving \eqref{eq:lscnnm}, numerous instances exceed the time limit or terminate due to insufficient memory (marked as 'oom' in the table), despite our desktop having 64GB of RAM. For instances that SPGL1 can solve, $\text{GLS}_\text{sec}$ is 10 to 100 times faster. Although ADMM outperforms $\text{GLS}_\text{sec}$ in Problem 1, the smallest instance in our test, $\text{GLS}_\text{sec}$ is 15 to 40 times faster for other instances that ADMM can solve. 

The comparison between the secant method and the bisection method is detailed in Table \ref{table:lrmc-sec-bis}. The secant method $\text{GLS}_\text{sec}$  is approximately twice as fast as the bisection variant $\text{GLS}_\text{bis}$ for most instances, and it requires significantly fewer iterations.
\begin{table}[!h]
	\centering
	\caption{The performance of proximal generation based level set method with secant iterations ($\text{GLS}_\text{sec}$), SPGL1 (SPG), and ADMM in solving the least-squares constrained nuclear norm minimization \eqref{eq:lscnnm} with the $\varrho=c\|b\|$ in Table \ref{table:lrmc_stat}.}	
	\begin{tabular}{ccccc}
		\hline
		\multirow{2}{*}{idx} & \multirow{2}{*}{rank} & time(s)                                                           &  & residual                              \\
		&                       & $\text{GLS}_\text{sec}$ : SPG : ADMM                              &  & $\text{GLS}_\text{sec}$ : SPG : ADMM  \\ \hline
		1                    & 60                    & 6.88+0 $\mid$  5.19+1 $\mid$  4.45+0                           &  & 2.33-4 $\mid$ 9.86-4 $\mid$ 9.54-4 \\
		2                    & 43                    & 1.05+1 $\mid$  2.02+2 $\mid$  1.63+2                           &  & 5.54-4 $\mid$ 4.34-4 $\mid$ 9.61-4 \\
		3                    & 32                    & 4.70+1 $\mid$  4.02+3 $\mid$  7.16+2                           &  & 2.68-4 $\mid$ 5.30-4 $\mid$ 2.10-4 \\
		4                    & 91                    & 2.96+2 $\mid$  \textgreater{}10800   $\mid$  8.17+3             &  & 4.04-4 $\mid$ 5.85-2 $\mid$ 5.61-4 \\
		5                    & 303                   & 6.14+3 $\mid$  oom $\mid$  oom                             &  & 4.16-5 $\mid$ nan $\mid$ nan         \\
		6                    & 139                   & 5.63+3 $\mid$  oom $\mid$  oom                             &  & 8.47-4 $\mid$ nan $\mid$ nan         \\
		7                    & 215                   & 1.18+4 $\mid$  oom $\mid$  oom                             &  & 2.03-4 $\mid$ nan $\mid$ nan         \\
		8                    & 150                   & 5.15+2 $\mid$  \textgreater{}10800   $\mid$  \textgreater{}10800 &  & 8.15-4 $\mid$ 5.14-1 $\mid$ 2.76+0 \\
		9                    & 153                   & 4.37+2 $\mid$  \textgreater{}10800   $\mid$  \textgreater{}10800 &  & 1.88-4 $\mid$ 5.14-1 $\mid$ 2.21+4 \\
		10                   & 131                   & 3.90+2 $\mid$  \textgreater{}10800   $\mid$  \textgreater{}10800 &  & 8.05-4 $\mid$ 5.54-1 $\mid$ 7.25-2 \\
		11                   & 65                    & 7.53+2 $\mid$  oom $\mid$  \textgreater{}10800                &  & 7.10-4 $\mid$ nan $\mid$ 2.28+4     \\
		12                   & 84                    & 7.99+1 $\mid$  \textgreater{}10800   $\mid$  2.94+3             &  & 8.10-4 $\mid$ 5.13-1 $\mid$ 6.53-4 \\
		13                   & 50                    & 2.26+2 $\mid$  3.96+3 $\mid$  6.08+3                           &  & 4.68-4 $\mid$ 4.02-4 $\mid$ 5.89-4 \\
		14                   & 50                    & 2.59+2 $\mid$  7.16+3 $\mid$  6.78+3                           &  & 1.43-4 $\mid$ 8.60-4 $\mid$ 7.91-4 \\
		15                   & 55                    & 4.11+2 $\mid$  9.90+3 $\mid$  9.48+3                           &  & 5.38-4 $\mid$ 3.63-1 $\mid$ 4.09-4 \\
		16                   & 60                    & 3.63+2 $\mid$  3.41+3 $\mid$  8.34+3                           &  & 7.02-4 $\mid$ 9.99-4 $\mid$ 3.16-4 \\
		17                   & 50                    & 4.08+2 $\mid$  \textgreater{}10800   $\mid$  \textgreater{}10800 &  & 6.24-4 $\mid$ 3.16-2 $\mid$ 1.71-3 \\
		18                   & 60                    & 5.38+2 $\mid$  \textgreater{}10800   $\mid$  \textgreater{}10800 &  & 9.76-4 $\mid$ 1.56-1 $\mid$ 9.50-3 \\
		19                   & 60                    & 5.57+2 $\mid$  \textgreater{}10800   $\mid$  \textgreater{}10800 &  & 9.73-4 $\mid$ 9.57-2 $\mid$ 1.15-2 \\
		20                   & 50                    & 1.07+3 $\mid$  \textgreater{}10800   $\mid$  \textgreater{}10800 &  & 6.69-4 $\mid$ 6.71-1 $\mid$ 1.13+5 \\
		21                   & 60                    & 5.00+3 $\mid$  oom $\mid$  oom                             &  & 6.10-6 $\mid$ nan $\mid$ nan         \\
		22                   & 60                    & 2.64+3 $\mid$  \textgreater{}10800   $\mid$  \textgreater{}10800 &  & 7.43-5 $\mid$ 2.68-1 $\mid$ 1.56+5 \\ \hline
	\end{tabular}
	\label{table:lrmc_test}
\end{table}

\begin{table}[!h]
			\centering
	\caption{The performance of proximal generation based level set method with secant iterations ($\text{GLS}_\text{sec}$) and with bisection iterations ($\text{GLS}_\text{bis}$)}
	\begin{tabular}{cccccccccc}
		\hline
		\multirow{2}{*}{idx} & \multicolumn{4}{c}{$\text{GLS}_\text{sec}$}              &  & \multicolumn{4}{c}{$\text{GLS}_\text{bis}$}              \\
		& $\lambda$      & iter & time(s)    & residual &  & $\lambda$      & iter & time(s)    & residual \\ \hline
		1                    & 1.016+1 & 5:1  & 6.88+0 & 2.33-4  &  & 1.016+1 & 9    & 1.21+1 & 2.28-4  \\
		2                    & 1.314+1 & 2:3  & 1.05+1 & 5.54-4  &  & 1.315+1 & 10   & 2.29+1 & 7.06-5  \\
		3                    & 3.015+1 & 3:3  & 4.70+1 & 2.68-4  &  & 3.014+1 & 10   & 9.54+1 & 9.30-4  \\
		4                    & 2.306+1 & 9:5  & 2.96+2 & 4.04-4  &  & 2.299+1 & 17   & 2.85+2 & 6.56-4  \\
		5                    & 3.857+1 & 7:3  & 6.14+3 & 4.16-5  &  & 3.855+1 & 9    & 6.43+3 & 6.77-4  \\
		6                    & 7.624+1 & 5:1  & 5.63+3 & 8.47-4  &  & 7.693+1 & 9    & 1.17+4 & 7.53-5  \\
		7                    & 7.741+1 & 4:2  & 1.18+4 & 2.03-4  &  & 7.860+1 & 7    & 2.26+4 & 8.73-4  \\
		8                    & 1.539+1 & 7:0  & 5.15+2 & 8.15-4  &  & 1.552+1 & 11   & 7.09+2 & 3.12-4  \\
		9                    & 1.496+1 & 6:0  & 4.37+2 & 1.88-4  &  & 1.496+1 & 6    & 4.32+2 & 1.88-4  \\
		10                   & 2.259+1 & 3:1  & 3.90+2 & 8.05-4  &  & 2.383+1 & 20   & 1.05+3 & 2.77-2  \\
		11                   & 3.632+1 & 7:5  & 7.53+2 & 7.10-4  &  & 3.673+1 & 20   & 1.36+3 & 6.57-4  \\
		12                   & 1.389+1 & 5:2  & 7.99+1 & 8.10-4  &  & 1.389+1 & 7    & 8.54+1 & 4.86-4  \\
		13                   & 7.252+1 & 2:1  & 2.26+2 & 4.68-4  &  & 7.257+1 & 10   & 4.50+2 & 5.47-4  \\
		14                   & 7.346+1 & 2:1  & 2.59+2 & 1.43-4  &  & 7.344+1 & 10   & 5.42+2 & 2.21-4  \\
		15                   & 8.196+1 & 2:1  & 4.11+2 & 5.38-4  &  & 8.194+1 & 10   & 8.98+2 & 8.13-4  \\
		16                   & 8.775+1 & 2:2  & 3.63+2 & 7.02-4  &  & 8.778+1 & 12   & 8.02+2 & 6.78-5  \\
		17                   & 5.455+1 & 2:3  & 4.08+2 & 6.24-4  &  & 5.440+1 & 13   & 8.59+2 & 6.48-4  \\
		18                   & 8.246+1 & 4:2  & 5.38+2 & 9.76-4  &  & 8.229+1 & 8    & 7.08+2 & 1.33-4  \\
		19                   & 8.268+1 & 4:2  & 5.57+2 & 9.73-4  &  & 8.283+1 & 9    & 7.94+2 & 2.88-4  \\
		20                   & 6.902+1 & 4:1  & 1.07+3 & 6.69-4  &  & 6.939+1 & 8    & 1.59+3 & 1.81-5  \\
		21                   & 8.239+1 & 6:2  & 5.00+3 & 6.10-6  &  & 8.235+1 & 10   & 6.23+3 & 3.50-4  \\
		22                   & 1.194+2 & 4:1  & 2.64+3 & 7.43-5  &  & 1.194+2 & 10   & 4.59+3 & 8.15-4  \\ \hline
	\end{tabular}
\label{table:lrmc-sec-bis}
\end{table}

We also compare our method with LoRADS \cite{Han2024}, a state-of-the-art SDP solver for matrix completion. Although LoRADS is designed for linear equality constraints, we constructed an equality-constrained counterpart using the optimal solution from $\text{GLS}_{\text{sec}}$ to enable a fair comparison.
Following their large-scale matrix completion setup, we use initial rank $\log m$, heuristic factor 2.5, phase 2 tolerance $10^{-3}$ and test four phase 1 tolerances ($10^{-3}$, $5\times10^{-3}$, $10^{-2}$ and $10^{-1}$) with a 3600s time limit. For fair comparison, we include the APG initialization time in $\text{GLS}_{\text{sec}}$'s total runtime since LoRADS cannot leverage this warm-start. Note that LoRADS deals with the SDP form, whose primal problem is $\min_{X\in\bS_+}\left\langle C,X \right\rangle \text{ s.t. } \cA X=b$ and dual problem is $\max_{\lambda}b^T\lambda\text{ s.t. }C-\cA^*\lambda\in\bS_+$.
 For evaluation, we compute the following primal feasibility, dual feasibility, and relative gap:
$$
\frac{\|\mathcal{A}(X)-b\|_2}{1+\|b\|_2}, \frac{\left|\min \left\{0, \sigma_{\min }\left(C-\mathcal{A}^*(\lambda)\right)\right\}\right|}{1+\|C\|_2}, \frac{\left|\langle C, X\rangle-\lambda^{\top} b\right|}{1+|\langle C, X\rangle|+\left|\lambda^{\top} b\right|}.
$$  
Table \ref{table:lorads} presents the best LoRADS results across different settings. The numerical results show that on our test problems, LoRADS encounters several challenges: (1) it cannot solve some larger-scale instances, (2) it requires substantially more computation time for solvable problems, and (3) it obtains solutions with relatively higher residuals within the time limit. Our method demonstrates consistent performance across all problem scales, achieving both computational efficiency and superior solution quality.
	\begin{table}[!h]
		\centering
		\caption{The performance of proximal generation based level set method with secant iterations ($\text{GLS}_\text{sec}$) and LoRADS}
	\begin{tabular}{cccc}
		\hline
		\multirow{2}{*}{idx} & $\text{GLS}_\text{sec}$                 &  & LoRADS                                   \\
		& time  $\mid$  residual &  & time   $\mid$   gap : prim : dual            \\ \hline
		3                      & 4.97+1   $\mid$   2.68-4 &  & 3.73+3  $\mid$  2.75-3 : 2.13-3 : NA      \\
		5                   & 6.21+3 $\mid$ 4.16-5 &  & out of memory                            \\
		4                     & 2.99+2 $\mid$ 4.04-4 &  & 3.93+3 $\mid$ 1.43-1 : 3.39-2 : NA      \\
		6                     & 5.83+3 $\mid$ 8.47-4 &  & out of memory                            \\
		7                  & 1.21+4 $\mid$ 2.03-4 &  & out of memory                            \\
		8                & 5.23+2 $\mid$ 8.15-4 &  & 2.81+3 $\mid$ 4.49-5 : 3.10-4 : 6.58-4 \\
		9                & 4.46+2 $\mid$ 1.88-4 &  & 2.56+3 $\mid$ 1.92-4 : 3.70-4 : 6.31-4 \\
		10                & 3.96+2 $\mid$ 8.05-4 &  & 3.22+3 $\mid$ 2.56-4 : 2.75-4 : 2.76-3 \\
		11                & 7.61+2 $\mid$ 7.10-4 &  & out of memory                            \\
		12                & 8.16+1 $\mid$ 8.10-4 &  & 9.74+2 $\mid$ 3.20-5 : 2.00-4 : 8.31-3 \\
		13                & 2.87+2 $\mid$ 4.68-4 &  & 3.87+3 $\mid$ 9.90-1 : 1.44+0 : NA      \\
		15                & 5.02+2 $\mid$ 5.38-4 &  & 4.30+3 $\mid$ 9.60-1 : 1.29+0 : NA      \\
		17                & 4.67+2 $\mid$ 6.24-4 &  & 4.11+3 $\mid$ 9.60-1 : 1.29+0 : NA      \\
		19                & 6.21+2 $\mid$ 9.73-4 &  & 4.79+3 $\mid$ 9.71-1 : 1.45+0 : NA      \\
		20                & 1.21+3 $\mid$ 6.69-4 &  & 4.94+3 $\mid$ 9.42-1 : 1.71+0 : NA      \\
		21                & 5.33+3 $\mid$ 6.10-6 &  & out of memory                            \\
		22               & 3.00+3 $\mid$ 7.43-5 &  & out of memory                            \\ \hline
	\end{tabular}
	\label{table:lorads}
\end{table}
\section{Conclusion and future work}

In this paper, we proposed a proximal generation based level set method with secant iterations for least-squares constrained nuclear norm minimization.
Theoretically, we established the nonsingularity of the Clarke generalized Jacobian for projection norm functions over closed convex sets. This fundamental result guarantees superlinear convergence of our secant method, achieving a fast convergence rate of $(1+\sqrt{5})/2$ under dual constraint nondegeneracy.
Computationally, we introduced a proximal generation method to solve the inner subproblems efficiently. By dynamically exploiting low-rank structures, our method significantly reduces the problem dimensionality. Extensive experiments demonstrate that our approach consistently outperforms state-of-the-art algorithms.

For future work, we plan to investigate the strong semismoothness of projection operators under conditions milder than nondegeneracy. We also aim to extend our algorithmic framework to broader classes of decomposed matrix optimization problems.

\section*{Declarations}
\medskip
\noindent\textbf{Data availability.}
The data that support the findings of this study are available from the corresponding author upon reasonable request.
\medskip

\noindent\textbf{Competing interests.}
The authors have no relevant financial or non-financial interests to disclose.



%
%



\end{document}